\documentclass[reqno]{amsart} 
\usepackage{amsmath,amsthm,amssymb,amsfonts}
\usepackage{mathrsfs}
\usepackage{longtable}
\numberwithin{equation}{section}
\usepackage{color}
\usepackage{enumerate}
\usepackage{multirow} 
\newtheoremstyle{mystyle}
{}
{}
{\normalfont}
{ }
{\bfseries}
{}
{10pt}
{ }
\theoremstyle{mystyle}
\newtheorem{theorem}{Theorem}
\newtheorem{proposition}{Proposition}
\newtheorem{lemma}{Lemma}
\newtheorem{remark}{Remark}

\usepackage{mathtools}
\usepackage[pdftex]{graphicx}
\usepackage{color} 

\def\Diag{\mathop{\rm Diag}\nolimits}
\def\tr{\mathop{\rm tr}\nolimits}
\def\vec{\mathop{\rm vec}\nolimits}
\def\vech{\mathop{\rm vech}\nolimits}
\def\det{\mathop{\rm det}\nolimits}
\def\rank{\mathop{\rm rank}\nolimits}

\allowdisplaybreaks[4]
\usepackage[foot]{amsaddr}
\usepackage[top=3cm, bottom=3cm, left=2.5cm, right=2.5cm]{geometry}
\usepackage[hang,small,bf]{caption}
\usepackage[subrefformat=parens]{subcaption}
\usepackage{comment}
\large
\title[QBIC of SEM for jump-diffusion processes]{QBIC of SEM for jump-diffusion processes based on high-frequency data}
\author[S. Kusano]{Shogo Kusano $^{1}$}
\author[M. Uchida]{Masayuki Uchida $^{2,3,4}$}
\address{$^{1}$ Faculty of Advanced Science and Technology, Kumamoto University, Kumamoto, Japan}
\address{$^{2}$ Graduate School of Engineering Science, The University of Osaka, Toyonaka, Japan}
\address{$^{3}$ Center for Mathematical Modeling and Data Science (MMDS), The University of Osaka, Toyonaka, Japan}
\address{$^{4}$ Graduate School of Mathematical Science, The University of Tokyo, Meguro, Japan and CREST, Japan Science and Technology Agency, Tokyo, Japan}
\begin{document}
\begin{abstract} 
\fontsize{9pt}{11pt}\selectfont
Structural equation modeling (SEM) is a statistical method for analyzing relationships among latent variables. Since SEM is a confirmatory method, the model needs to be specified in advance. In practice, however, statisticians have several candidate models and aim to select the most appropriate one among them. In this paper, we consider model selection in SEM for jump-diffusion processes. We propose a quasi-Bayesian information criterion (QBIC) for the SEM and show that the proposed criterion has model-selection consistency.
\end{abstract}
\maketitle

\section{Introduction} 
\fontsize{11pt}{18pt}\selectfont
\raggedbottom
\setlength{\abovedisplayskip}{8pt}
\setlength{\belowdisplayskip}{8pt}
We study model selection in structural equation modeling (SEM) for jump-diffusion processes. Let $(\Omega, \mathcal{F},(\mathcal{F}_t)_{t\geq 0}, {\bf{P}})$ be a filtered probability space. First of all, we introduce the true model. The $p_1$-dimensional observable process $\{X_{1,0,t}\}_{t\geq 0}$ is given by the following true factor model:
\begin{align*}
    X_{1,0,t}&={\bf{\Lambda}}_{1,0}\xi_{0,t}+\delta_{0,t},
\end{align*}
where $\{\xi_{0,t}\}_{t\geq 0}$ and $\{\delta_{0,t}\}_{t\geq 0}$ are $k_1$ and $p_1$-dimensional c{\`a}dl{\`a}g $(\mathcal{F}_t)$-adapted latent processes, respectively, $k_1\leq p_1$ and ${\bf{\Lambda}}_{1,0}\in\mathbb{R}^{p_1\times k_1}$ is a true constant loading matrix. The stochastic processes $\{\xi_{0,t}\}_{t\geq 0}$ and $\{\delta_{0,t}\}_{t\geq 0}$ are 
governed by the stochastic differential equations:
\begin{align*}
    d\xi_{0,t}&=a_{1}(\xi_{0,t-})dt+{\bf{S}}_{1,0}d W_{1,t}+\int_{E_1}c_1(\xi_{0,t-},z_1)p_1(dt,dz_1),\quad \xi_{0,0}=x_{1,0}
\end{align*}
and
\begin{align*}
    d\delta_{0,t}&=a_{2}(\delta_{0,t-})dt+{\bf{S}}_{2,0}d W_{2,t}+\int_{E_2}c_2(\delta_{0,t-},z_2)p_2(dt,dz_2),\quad \delta_{0,0}=x_{2,0},
\end{align*}
where $E_1=\mathbb{R}^{k_1} \backslash \{0\}$, $E_2=\mathbb{R}^{p_1} \backslash \{0\}$, 
$a_1:\mathbb{R}^{k_1}\longrightarrow\mathbb{R}^{k_1}$, $a_2:\mathbb{R}^{p_1}\longrightarrow\mathbb{R}^{p_1}$, 
$c_1:\mathbb{R}^{k_1}\times E_1\longrightarrow\mathbb{R}^{k_1}$ and
$c_2:\mathbb{R}^{p_1}\times E_2\longrightarrow\mathbb{R}^{p_1}$ are unknown Borel functions, ${\bf{S}}_{1,0}\in\mathbb{R}^{k_1\times r_1}$ and ${\bf{S}}_{2,0}\in\mathbb{R}^{p_1\times r_2}$ are true constant diffusion coefficient matrices, 
the initial values $x_{1,0}\in\mathbb{R}^{k_1}$ and $x_{2,0}\in\mathbb{R}^{p_1}$ are non-random vectors, $\{W_{1,t}\}_{t\geq 0}$ and $\{W_{2,t}\}_{t\geq 0}$ are $r_1$ and $r_2$-dimensional standard $(\mathcal{F}_t)$-Wiener processes, and $p_1(dt,dz_1)$ and $p_2(dt,dz_2)$ are Poisson random measures on $\mathbb{R}_{+}\times E_1$ and $\mathbb{R}_{+}\times E_2$
with compensators $q_1(dt,dz_1)={\bf{E}}\bigl[p_1(dt,dz_1)\bigr]$ and $q_2(dt,dz_2)={\bf{E}}\bigl[p_2(dt,dz_2)\bigr]$. Here, ${\bf E}$ denotes the expectation operator. The $p_2$-dimensional observable process $\{X_{2,0,t}\}_{t\geq 0}$ is defined by the following true factor model:
\begin{align*}
    X_{2,0,t}&={\bf{\Lambda}}_{2,0}\eta_{0,t}+\varepsilon_{0,t},
\end{align*}
where $\{\eta_{0,t}\}_{t\geq 0}$ and $\{\varepsilon_{0,t}\}_{t\geq 0}$ are $k_2$ and $p_2$-dimensional c{\`a}dl{\`a}g $(\mathcal{F}_t)$-adapted latent processes, respectively, $k_2\leq p_2$ and ${\bf{\Lambda}}_{2,0}\in\mathbb{R}^{p_2\times k_2}$ is a true constant loading matrix. The stochastic process $\{\varepsilon_{0,t}\}_{t\geq 0}$ is 
governed by
the stochastic differential equation:
\begin{align*}
    d\varepsilon_{0,t}&=a_{3}(\varepsilon_{0,t-})dt+{\bf{S}}_{3,0}d W_{3,t}+\int_{E_3}c_3(\varepsilon_{0,t-},z_3)p_3(dt,dz_3),\quad \varepsilon_{0,0}=x_{3,0},
\end{align*}
where $E_3=\mathbb{R}^{p_2} \backslash \{0\}$, $a_3:\mathbb{R}^{p_2}\longrightarrow\mathbb{R}^{p_2}$ and $c_3:\mathbb{R}^{p_2}\times E_3\longrightarrow\mathbb{R}^{p_2}$ are unknown Borel functions, ${\bf{S}}_{3,0}\in\mathbb{R}^{p_2\times r_3}$ is a true constant diffusion coefficient matrix, the initial value $x_{3,0}\in\mathbb{R}^{p_2}$ is a non-random vector, $\{W_{3,t}\}_{t\geq 0}$ is an $r_3$-dimensional standard $(\mathcal{F}_t)$-Wiener process, and $p_3(dt,dz_3)$ is a Poisson random measure on $\mathbb{R}_{+}\times E_3$ with a 
compensator $q_3(dt,dz_3)={\bf{E}}\bigl[p_3(dt,dz_3)\bigr]$. In addition, we assume that the relationship between $\{\xi_{0,t}\}_{t\geq 0}$ and $\{\eta_{0,t}\}_{t\geq 0}$ is given by
\begin{align*}
    \eta_{0,t}={\bf{B}}_0\eta_{0,t}+{\bf{\Gamma}}_0\xi_{0,t}+\zeta_{0,t},
\end{align*}
where $\{\zeta_{0,t}\}_{t\geq 0}$ is a $k_2$-dimensional c{\`a}dl{\`a}g $(\mathcal{F}_t)$-adapted latent process, ${\bf{B}}_0\in\mathbb{R}^{k_2\times k_2}$ is a true constant loading matrix whose diagonal elements are zero, and ${\bf{\Gamma}}_0\in\mathbb{R}^{k_2\times k_1}$ is a true constant loading matrix.
The stochastic process $\{\zeta_{0,t}\}_{t\geq 0}$ is assumed to satisfy the following stochastic differential equation:
\begin{align*}
    d\zeta_{0,t}&=a_{4}(\zeta_{0,t-})dt+{\bf{S}}_{4,0}d W_{4,t}+\int_{E_4}c_4(\zeta_{0,t-},z_4)p_4(dt,dz_4),\quad \zeta_{0,0}=x_{4,0}, 
\end{align*}
where $E_4=\mathbb{R}^{k_2} \backslash \{0\}$, $a_4:\mathbb{R}^{k_2}\longrightarrow\mathbb{R}^{k_2}$ and $c_4:\mathbb{R}^{k_2}\times E_4\longrightarrow\mathbb{R}^{k_2}$ are unknown Borel functions, ${\bf{S}}_{4,0}\in\mathbb{R}^{k_2\times r_4}$ is a true constant diffusion coefficient matrix, the initial value $x_{4,0}\in\mathbb{R}^{k_2}$ is a non-random vector, $\{W_{4,t}\}_{t\geq 0}$ is an $r_4$-dimensional standard $(\mathcal{F}_t)$-Wiener process, and $p_4(dt,dz_4)$ is a Poisson random measure on $\mathbb{R}_{+}\times E_4$ with a compensator $q_4(dt,dz_4)={\bf{E}}\bigl[p_4(dt,dz_4)\bigr]$. 
We define the volatility matrices of $\{\xi_{0,t}\}_{t\geq 0}$, $\{\delta_{0,t}\}_{t\geq 0}$, $\{\varepsilon_{0,t}\}_{t\geq 0}$ and $\{\zeta_{0,t}\}_{t\geq 0}$ as ${\bf{\Sigma}}_{\xi\xi,0}={\bf{S}}_{1,0}{\bf{S}}_{1,0}^{\top}$, 
${\bf{\Sigma}}_{\delta\delta,0}={\bf{S}}_{2,0}{\bf{S}}_{2,0}^{\top}$, ${\bf{\Sigma}}_{\varepsilon\varepsilon,0}={\bf{S}}_{3,0}{\bf{S}}_{3,0}^{\top}$ and 
${\bf{\Sigma}}_{\zeta\zeta,0}={\bf{S}}_{4,0}{\bf{S}}_{4,0}^{\top}$, respectively, where $\top$ denotes matrix transpose. We assume that ${\bf{\Lambda}}_{1,0}$ and ${\bf{\Lambda}}_{2,0}$ are of full column rank, and that ${\bf{\Psi}}_0=\mathbb{I}_{k_2}-{\bf{B}}_0$ is non-singular, where $\mathbb{I}_{k_2}$ is the identity matrix of size $k_2$. It is supposed that for any $t\geq 0$,
\begin{align*}
    \mathcal{F}_t,\quad \sigma(W_{i,u}-W_{i,t};\ u\geq t) \ \  (i=1,2,3,4)
\end{align*}
and
\begin{align*}
    \sigma\bigl(p_i(A_i\cap ((t,\infty)\times E_i)); A_i\subset\mathbb{R}_{+}\times E_i \mbox{\ is a Borel set}\bigr)\quad (i=1,2,3,4)
\end{align*}
are independent. For each $i=1,2,3,4$, we assume that the compensator $q_i(dt,dz_i)$ admits the representation $q_i(dt,dz_i)=f_i(z_i)dz_idt$ and $f_i(z_i)=\lambda_{i,0}F_i(z_i)$, where $\lambda_{i,0}>0$ is an unknown value, and $F_i(z_i)$ is an unknown probability density. Set $p=p_1+p_2$ and $X_{0,t}=(X_{1,0,t}^{\top},X_{2,0,t}^{\top})^{\top}$. To simplify notation, 
We write
$X_{0,t}$ as $X_t$. $\{X_{t_i^n}\}_{i=0}^n$ are discrete observations, where $t_i^n=ih_n$ and $T=nh_n$ is fixed. Let
\begin{align*}
    {\bf{\Sigma}}_0=\begin{pmatrix}
    {\bf{\Sigma}}_0^{11} & {\bf{\Sigma}}_0^{12}\\
    {\bf{\Sigma}}_0^{12\top} & {\bf{\Sigma}}_0^{22}
    \end{pmatrix},
\end{align*}
where
\begin{align*}
    {\bf{\Sigma}}_0^{11}&={\bf{\Lambda}}_{1,0}{\bf{\Sigma}}_{\xi\xi,0}{\bf{\Lambda}}_{1,0}^{\top}+{\bf{\Sigma}}_{\delta\delta,0},\\
    {\bf{\Sigma}}_0^{12}&={\bf{\Lambda}}_{1,0}{\bf{\Sigma}}_{\xi\xi,0}{\bf{\Gamma}}_0^{\top}{\bf{\Psi}}_0^{-1\top}{\bf{\Lambda}}_{2,0}^{\top},\\
    {\bf{\Sigma}}_0^{22}&={\bf{\Lambda}}_{2,0}{\bf{\Psi}}_0^{-1}({\bf{\Gamma}}_0{\bf{\Sigma}}_{\xi\xi,0}{\bf{\Gamma}}_0^{\top}+{\bf{\Sigma}}_{\zeta\zeta,0}){\bf{\Psi}}_0^{-1\top}{\bf{\Lambda}}_{2,0}^{\top}+{\bf{\Sigma}}_{\varepsilon\varepsilon,0}.
\end{align*}
We assume that ${\bf{\Sigma}}_0$ is positive definite.  Next, we introduce parametric models. For Model $m\in\{1,\ldots,M\}$, let $\theta_m\in\Theta_m\subset\mathbb{R}^{q_m}$ denote the parameter vector, where $\Theta_m$ is a bounded open convex subset of $\mathbb{R}^{q_m}$. Here, $\theta_m$ contains only the unknown components of the loading matrices and the volatility matrices of the latent processes in Model $m$. 
Some examples of such parameters are given in Section \ref{simulation}. The $p_1$-dimensional observable process $\{X^{\theta}_{1,m,t}\}_{t\geq 0}$ is defined by the following factor model:
\begin{align*}
    X^{\theta}_{1,m,t}={\bf{\Lambda}}^{\theta}_{1,m}\xi^{\theta}_{m,t}+\delta^{\theta}_{m,t}, 
\end{align*}
where $\{\xi^{\theta}_{m,t}\}_{t\geq 0}$ and $\{\delta^{\theta}_{m,t}\}_{t\geq 0}$
are $k_1$ and $p_1$-dimensional c{\`a}dl{\`a}g $(\mathcal{F}_t)$-adapted latent processes, respectively, and ${\bf{\Lambda}}^{\theta}_{1,m}\in\mathbb{R}^{p_1\times k_1}$ is a constant loading matrix. The stochastic processes $\{\xi^{\theta}_{m,t}\}_{t\geq 0}$ and $\{\delta^{\theta}_{m,t}\}_{t\geq 0}$ are supposed to satisfy the stochastic differential equations:
\begin{align*}
    d\xi_{m,t}^{\theta}&=a_{1}(\xi^{\theta}_{m,t-})dt+{\bf{S}}^{\theta}_{1,m}d W_{1,t}+\int_{E_1}c_1(\xi^{\theta}_{m,t-},z_1)p_1(dt,dz_1),\quad \xi^{\theta}_{m,0}=x_{1,0}
\end{align*}
and
\begin{align*}
    d\delta_{m,t}^{\theta}&=a_{2}(\delta^{\theta}_{m,t-})dt+{\bf{S}}^{\theta}_{2,m}d W_{2,t}+\int_{E_2}c_2(\delta^{\theta}_{m,t-},z_2)p_2(dt,dz_2),\quad \delta^{\theta}_{m,0}=x_{2,0},
\end{align*}
where ${\bf{S}}^{\theta}_{1,m}\in\mathbb{R}^{k_1\times r_1}$ and ${\bf{S}}^{\theta}_{2,m}\in\mathbb{R}^{p_1\times r_2}$ are constant diffusion coefficient matrices. The $p_2$-dimensional observable process $\{X^{\theta}_{2,m,t}\}_{t\geq 0}$ is given by the following factor model:
\begin{align*}
    X^{\theta}_{2,m,t}&={\bf{\Lambda}}^{\theta}_{2,m}\eta^{\theta}_{m,t}+\varepsilon^{\theta}_{m,t},
\end{align*}
where $\{\eta^{\theta}_{m,t}\}_{t\geq 0}$ and $\{\varepsilon^{\theta}_{m,t}\}_{t\geq 0}$ are $k_2$ and $p_2$-dimensional c{\`a}dl{\`a}g $(\mathcal{F}_t)$-adapted latent processes, respectively, and ${\bf{\Lambda}}^{\theta}_{2,m}\in\mathbb{R}^{p_2\times k_2}$ is a constant loading matrix. The stochastic process $\{\varepsilon^{\theta}_{m,t}\}_{t\geq 0}$ is 
governed by the stochastic differential equation:
\begin{align*}
    d\varepsilon_{m,t}^{\theta}&=a_{3}(\varepsilon^{\theta}_{m,t-})dt+{\bf{S}}^{\theta}_{3,m}d W_{3,t}+\int_{E_3}c_3(\varepsilon^{\theta}_{m,t-},z_3)p_3(dt,dz_3),\quad \varepsilon^{\theta}_{m,0}=x_{3,0},
\end{align*}
where ${\bf{S}}^{\theta}_{3,m}\in\mathbb{R}^{p_2\times r_3}$ is a constant diffusion coefficient matrix. Furthermore, we assume that the relationship between $\{\xi^{\theta}_{m,t}\}_{t\geq 0}$ and $\{\eta^{\theta}_{m,t}\}_{t\geq 0}$ is defined by
\begin{align*}
    \eta^{\theta}_{m,t}={\bf{B}}_m^{\theta}\eta^{\theta}_{m,t}+{\bf{\Gamma}}_m^{\theta}\xi^{\theta}_{m,t}+\zeta^{\theta}_{m,t},
\end{align*}
where $\{\zeta^{\theta}_{m,t}\}_{t\geq 0}$ is a $k_2$-dimensional c{\`a}dl{\`a}g $(\mathcal{F}_t)$-adapted latent process, ${\bf{B}}_m^{\theta}\in\mathbb{R}^{k_2\times k_2}$ is a constant loading matrix whose diagonal elements are zero, and ${\bf{\Gamma}}_m^{\theta}\in\mathbb{R}^{k_2\times k_1}$ is a constant loading matrix.
The stochastic process $\{\zeta^{\theta}_{m,t}\}_{t\geq 0}$ is assumed to satisfy the following stochastic differential equation:
\begin{align*}
    d\zeta_{m,t}^{\theta}&=a_{4}(\zeta^{\theta}_{m,t-})dt+{\bf{S}}^{\theta}_{4,m}d W_{4,t}+\int_{E_4}c_4(\zeta^{\theta}_{m,t-},z_4)p_4(dt,dz_4),\quad \zeta^{\theta}_{m,0}=x_{4,0},
\end{align*}
where ${\bf{S}}^{\theta}_{4,m}\in\mathbb{R}^{k_2\times r_4}$ is a constant diffusion coefficient matrix. It is supposed that ${\bf{\Lambda}}^{\theta}_{1,m}$ and ${\bf{\Lambda}}^{\theta}_{2,m}$ are of full column rank, and
${\bf{\Psi}}^{\theta}_m=\mathbb{I}_{k_2}-{\bf{B}}_m^{\theta}$ is non-singular. The volatility matrices of $\{\xi^{\theta}_{m,t}\}_{t\geq 0}$, $\{\delta^{\theta}_{m,t}\}_{t\geq 0}$, $\{\varepsilon^{\theta}_{m,t}\}_{t\geq 0}$ and $\{\zeta^{\theta}_{m,t}\}_{t\geq 0}$ are defined by ${\bf{\Sigma}}^{\theta}_{\xi\xi,m}={\bf{S}}^{\theta}_{1,m}{\bf{S}}_{1,m}^{\theta\top}$, 
${\bf{\Sigma}}^{\theta}_{\delta\delta,m}={\bf{S}}^{\theta}_{2,m}{\bf{S}}_{2,m}^{\theta\top}$, ${\bf{\Sigma}}^{\theta}_{\varepsilon\varepsilon,m}={\bf{S}}^{\theta}_{3,m}{\bf{S}}_{3,m}^{\theta\top}$ and 
${\bf{\Sigma}}^{\theta}_{\zeta\zeta,m}={\bf{S}}^{\theta}_{4,m}{\bf{S}}_{4,m}^{\theta\top}$, respectively. Set $X_{m,t}^{\theta}=(X_{1,m,t}^{\theta\top},X_{2,m,t}^{\theta\top})^{\top}$ and
\begin{align*}
    {\bf{\Sigma}}_m(\theta_m)=\begin{pmatrix}
    {\bf{\Sigma}}_m^{11}(\theta_m) & {\bf{\Sigma}}_m^{12}(\theta_m)\\
    {\bf{\Sigma}}_m^{12}(\theta_m)^{\top} & {\bf{\Sigma}}_m^{22}(\theta_m)
    \end{pmatrix},
\end{align*}
where 
\begin{align*}                      
    \qquad\qquad{\bf{\Sigma}}^{11}_m(\theta_m)&
    ={\bf{\Lambda}}^{\theta}_{1,m}{\bf{\Sigma}}^{\theta}_{\xi\xi,m}{\bf{\Lambda}}_{1,m}^{\theta\top}
    +{\bf{\Sigma}}^{\theta}_{\delta\delta,m},\\  {\bf{\Sigma}}^{12}_m(\theta_m)&={\bf{\Lambda}}^{\theta}_{1,m}{\bf{\Sigma}}^{\theta}_{\xi\xi,m}{\bf{\Gamma}}_m^{\theta\top}{\bf{\Psi}}_m^{\theta-1\top}{\bf{\Lambda}}_{2,m}^{\theta\top},\\
    {\bf{\Sigma}}^{22}_m(\theta_m)&={\bf{\Lambda}}^{\theta}_{2,m}{\bf{\Psi}}_m^{\theta-1}({\bf{\Gamma}}_m^{\theta}{\bf{\Sigma}}^{\theta}_{\xi\xi,m}{\bf{\Gamma}}_m^{\theta\top}+{\bf{\Sigma}}^{\theta}_{\zeta\zeta,m}){\bf{\Psi}}_m^{\theta-1\top}{\bf{\Lambda}}_{2,m}^{\theta\top}+{\bf{\Sigma}}^{\theta}_{\varepsilon\varepsilon,m}.
\end{align*}
We assume that there exists $\theta_{m,0}\in \Theta_m$ such that ${\bf{\Sigma}}_0={\bf{\Sigma}}_m(\theta_{m,0})$. Moreover, it is assumed that ${\bf{\Sigma}}_m(\theta_m)$ for all $\theta_m\in \bar{\Theta}_m$ is positive definite.

SEM is a statistical method for investigating relationships among latent variables. For more details on SEM, we refer readers to Shapiro \cite{Shapiro(1983),Shapiro(1985)}, Everitt \cite{Everitt(1984)}, and Mueller \cite{Mueller(1999)}. Since SEM is a confirmatory analysis method, 
the model needs to be specified in advance according to the theoretical framework of the relevant research field. In practice, however, one often has several candidate SEMs and 
it is necessary to select the most appropriate model among them. For this reason, a number of information criteria for SEM have been developed; see, for example, Huang \cite{Huang AIC(2017)}.
In particular, the Akaike information criterion (AIC) and the Bayesian information criterion (BIC) are commonly used for model selection. Note that these criteria are motivated by different purposes. AIC aims to select the model that minimizes the Kullback--Leibler divergence from a predictive point of view, whereas BIC aims to select the model with the largest posterior probability given the data. The most important difference between AIC and BIC is whether they have model-selection consistency. It is well known that BIC has model-selection consistency, whereas AIC does not.

SEM for time-series data has also been studied; see, e.g,  Czir{\'a}ky \cite{Cziraky(2004)}.
In recent years, we can easily access high-frequency data thanks to the development of information technology. In particular, high-frequency data are commonly observed in finance, and statistical inference for continuous-time stochastic processes based on high-frequency data
has been investigated; see, e.g., Yoshida \cite{Yoshida(1992)}, Genon-Catalot and Jacod \cite{Genon(1993)}, Kessler \cite{kessler(1997)}, 
Uchida and Yoshida \cite{Uchi-Yoshi(2012)}, and references therein. We also note that the information criterion for continuous-time stochastic processes based on high-frequency data has also been extensively examined; see, e.g., Uchida \cite{Uchida(2010)}, Fujii and Uchida \cite{Fuji(2014)}, Eguchi and Masuda \cite{Eguchi(2018), Eguchi(2024), Eguchi(2026)}, Eguchi and Uehara \cite{Eguchi(2021)}, and Uehara \cite{Uehara(2025)}. Recently, Kusano and Uchida \cite{Kusano(JJSD)} studied SEM for diffusion processes based on high-frequency data, and information criteria for the 
SEM were developed by Kusano and Uchida \cite{Kusano(AIC), Kusano(BIC)}. These studies enable us to investigate relationships among latent continuous-time stochastic processes based on high-frequency data such as stock prices. On the other hand, the discontinuous sample paths are frequently observed in practice. To analyze such data, modeling for jump-diffusion processes
is useful, and many researchers have investigated statistical inference for jump-diffusion processes based on high-frequency data; see, e.g., Shimizu and Yoshida \cite{Shimizu-Yoshida_JP}, Mancini \cite{Mancini(2004)}, Shimizu and Yoshida \cite{Shimizu(2006)}, Ogihara and Yoshida \cite{Ogihara(2011)}, Inatsugu and Yoshida \cite{Inatsugu(2021)}, and Amorino and Gloter \cite{Amorino(2021)}. To conduct SEM based on high-frequency data with discontinuous sample paths, Kusano and Uchida \cite{Kusano(jump)} considered SEM for jump-diffusion processes, and Kusano and Uchida \cite{Kusano(jumpAIC)} introduced the quasi-Akaike information criterion (QAIC) of the SEM. However, as shown in Theorem \ref{QAIC} of this paper, the QAIC lacks model-selection consistency. Therefore, in this paper, we propose a quasi-Bayesian information criterion (QBIC) of SEM for jump-diffusion processes and show that the criterion has model-selection consistency.

The structure of this paper is as follows. In Section \ref{notation}, we introduce some notation and assumptions. In Section \ref{main}, we propose a quasi-Bayesian
information criterion of SEM for jump-diffusion processes using an asymptotic expansion of the marginal quasi-log-likelihood, and also show that the proposed information criterion has model-selection consistency. Section \ref{simulation} is devoted to examples and simulation studies. In Section \ref{proofs}, we provide the proofs of the results presented in Section \ref{main}.
\section{Notation and Assumptions}\label{notation}
We first introduce the notation used throughout the paper. For a vector $v$, 
$v^{(i)}$ denotes its $i$-th component, $\Diag v$ denotes the diagonal matrix whose 
$i$-th diagonal entry is $v^{(i)}$, and $|v|$ is defined by
\begin{align*}
    |v|=\sqrt{\tr(vv^\top)}.
\end{align*}
For a matrix $M$, $M_{ij}$ denotes its $(i,j)$-th entry, and we define
\begin{align*}
    |M|=\sqrt{\tr(MM^\top)}.
\end{align*}
For a symmetric matrix $A$, $\lambda_{min}$ is the smallest eigenvalue of $A$. For a symmetric matrix $M$, $\vec M$ and $\vech M$ denote its vectorization and 
half-vectorization, respectively. Let 
$\mathcal{M}_p^{+}$ denote the set of all $p\times p$ real positive definite matrices. 
Let $C_{\uparrow}^k(\mathbb{R}^d)$ be the set of all functions on $\mathbb{R}^d$ 
that are continuously differentiable up to order $k$ and whose derivatives are of 
polynomial growth. The symbols $\stackrel{p}{\longrightarrow}$ and 
$\stackrel{d}{\longrightarrow}$ denote convergence in probability and convergence in 
distribution, respectively. A $d$-dimensional normal random variable with mean 
$\mu\in\mathbb{R}^d$ and covariance matrix $\Sigma\in\mathcal{M}_d^{+}$ is denoted by 
$N_d(\mu,\Sigma)$. $Z_d$ stands for a  $d$-dimensional standard normal random variable. $\chi^2_{r}$ represents the random variable which has the chi-squared distribution with $r>0$ degrees of freedom. For a stochastic process $\{S_t\}_{t\geq0}$, we write
\begin{align*}
    \Delta_i^n S = S_{t_i^n}-S_{t_{i-1}^n}.
\end{align*}
Set $\partial_{\theta}=\partial/\partial\theta$ and $\partial^2_{\theta}=\partial_{\theta}\partial_{\theta}^{\top}$. Let $B(\theta_{m,0},\rho)$ denote the closed ball in $\mathbb{R}^{q_m}$ with center $\theta_{m,0}$ and radius
$\rho>0$. Let $\pi_{m}(\theta_{m})$ be the prior density of Model $m$. Next, we state the assumptions. Set $(d_1,d_2,d_3,d_4)=(k_1,p_1,p_2,k_2)$.
\begin{enumerate}
    \vspace{2mm}
    \item[\bf{[A1]}]\ 
    For $i=1,2,3,4$, there exist $L_i>0$ and $\zeta_i(z_i)>0$ of at most polynomial growth in $z_i$ such that
    \begin{align*}
        |a_i(x_i)-a_i(y_i)|\leq L_i|x_i-y_i|,\ |c_i(x_i,z_i)-c_i(y_i,z_i)|\leq \zeta_i(z_i)|x_i-y_i|
    \end{align*}
    and
    \begin{align*}
        |c_i(x_i,z_i)|\leq \zeta_i(z_i)(1+|x_i|)
    \end{align*}
    for all $x_i,y_i\in \mathbb{R}^{d_i}$ and $z_i\in E_i$.
    \vspace{2mm}
    \item[\bf{[A2]}]\ 
     For $i=1,2,3,4$, $a_i\in C_{\uparrow}^4(\mathbb{R}^{d_i})$.
     \vspace{2mm}
    \item[\bf{[A3]}] \ For $i=1,2,3,4$, there exist $r_i>0$ and $K_i>0$ such that
    \begin{align*}
        f_i(z_i){\bf{1}}_{\{|z_i|\leq r_i\}}\leq K_i|z_i|^{1-d_i}
    \end{align*}
    and
    \begin{align*}
        \int |z_i|^{L}f_i(z_i)dz_i<\infty
    \end{align*}
    for all $L\geq 1$.
    \vspace{2mm}
    \item[\bf{[A4]}]\ 
    For $i=1,2,3,4$, it follows that
    \begin{align*}
        \inf_{x_i}|c_i(x_i,z_i)|\geq c_{i,0}|z_i|
    \end{align*}
    for some $c_{i,0}>0$ near the origin.
    \vspace{2mm}
    \item[\bf{[B]}]\ $\pi_m(\theta_m)$ is continuous and bounded in $\bar{\Theta}_m$, and $\pi_m(\theta_{m,0})>0$.
\end{enumerate}
\begin{remark}
{\bf{[A1]}}-{\bf{[A4]}} are the standard assumptions for jump-diffusion processes; see Kusano and Uchida \cite{Kusano(jump)} for further details. See Eguchi and Masuda \cite{Eguchi(2024)} for more details of {\bf{[B]}}.
\end{remark}
\section{Main results}\label{main}
In this paper, we judge whether a jump occurs in the interval $(t_{i-1}^n,t_i^n]$ from the size of the increment
$|\Delta_i^n X|$; see, for example, Shimizu and Yoshida \cite{Shimizu(2006)}, Ogihara and Yoshida \cite{Ogihara(2011)} and Kusano and Uchida \cite{Kusano(jump)}. More precisely, we judge that a jump occurs in the interval if $|\Delta_i^n X|>Dh_n^{\rho}$ 
or that no jump occurs otherwise, where $\rho\in[0,1/2)$ and $D>0$. By this truncation method, we define the quasi-likelihood as follows:
\begin{align*}
    {\bf L}_{m,n}(\theta_m)=\exp\{{\bf H}_{m,n}(\theta_m)\},
\end{align*}
where
\begin{align*}
    {\bf H}_{m,n}(\theta_m)
    &=
    -\frac{1}{2h_n}\sum_{i=1}^n
    (\Delta_i^n X)^\top{\bf\Sigma}_m(\theta_m)^{-1}(\Delta_i^n X)
    {\bf 1}_{\{|\Delta_i^n X|\le Dh_n^\rho\}}\\
    &\qquad
    -\frac{1}{2}\sum_{i=1}^n
    \log\det{\bf\Sigma}_m(\theta_m)\,
    {\bf 1}_{\{|\Delta_i^n X|\le Dh_n^\rho\}}.
\end{align*}
We refer to Kusano and Uchida \cite{Kusano(jump)} for further details on this quasi-likelihood. Here, we assume that $\rho\in[1/3,1/2)$.
The quasi-likelihood estimator \(\hat{\theta}_{m,n}\) is given by
\begin{align*}
    {\bf H}_{m,n}(\hat{\theta}_{m,n})
    =
    \sup_{\theta_m\in\bar{\Theta}_m}{\bf H}_{m,n}(\theta_m).
\end{align*}
It should be noted that \({\bf H}_{m,n}(\theta_m)\) is three times continuously differentiable with respect to \(\theta_m\), and that $\partial_{\theta_m}{\bf H}_{m,n}(\theta_m)$, $\partial_{\theta_m}^2{\bf H}_{m,n}(\theta_m)$, and $\partial_{\theta_m}^3{\bf H}_{m,n}(\theta_m)$ admit continuous extensions to the boundary $\partial\Theta_m$. A Schwarz-type information criterion aims to select the model that maximizes the marginal quasi-likelihood given by
\begin{align*}
    \int_{\Theta_m}\exp\bigl\{{\bf{H}}_{m,n}(\theta_m)\bigr\}\pi_m(\theta_m)d\theta_m.
\end{align*}
For details, see Kusano and Uchida \cite{Kusano(BIC)}.
Since the marginal quasi-likelihood is generally intractable, we consider its asymptotic expansion. Let
\begin{align*}
    {\bf{Y}}_m(\theta_m)&={\bf{H}}_m(\theta_m)-{\bf{H}}_m(\theta_{m,0})\\
    &=-\frac{1}{2}\tr
    \biggl\{\Bigl({\bf{\Sigma}}_m(\theta_m)^{-1}-{\bf{\Sigma}}_m(\theta_{m,0})^{-1}\Bigr){\bf{\Sigma}}_m(\theta_{m,0})\biggr\}-\frac{1}{2}
    \log\frac{\det{\bf{\Sigma}}_m(\theta_m)}{\det{\bf{\Sigma}}_m(\theta_{m,0})},
\end{align*}
where
\begin{align*}
    {\bf{H}}_m(\theta_m)=-\frac{1}{2}\tr\Bigl\{{\bf{\Sigma}}_m(\theta_m)^{-1}{\bf{\Sigma}}_m(\theta_{m,0})\Bigr\}-\frac{1}{2}\log\det {\bf{\Sigma}}_m(\theta_m).
\end{align*}
In addition, we impose the following assumption.
\begin{enumerate}
    \vspace{2mm}
    \item[\bf{[C1]}]There exists a constant $\chi>0$ such that
    \begin{align*}
        {\bf{Y}}_m(\theta_m)\leq -\chi|\theta_m-\theta_{m,0}|^2
    \end{align*}
    for all $\theta_m\in\Theta_m$.
\end{enumerate}
\begin{remark}
{\bf{[C1]}} is an identifiability condition. If 
\begin{align*}
    \left.\Delta_{m,0}=\frac{\partial}{\partial\theta_m^{\top}}\vech{{\bf{\Sigma}}_m(\theta_m)}\right|_{\theta_m=\theta_{m,0}}
\end{align*}
is full-column rank and
\begin{align*}
    {\bf{\Sigma}}_m(\theta_m)={\bf{\Sigma}}_m(\theta_{m,0})\Longrightarrow \theta_m=\theta_{m,0},
\end{align*}
then {\bf{[C1]}} holds.
\end{remark}
The marginal quasi-log-likelihood can be expanded asymptotically as follows.
\begin{theorem}\label{Hexpansion}
Under $\bf{[A1]}$-$\bf{[A4]}$, $\bf{[B]}$ and $\bf{[C1]}$, 
\begin{align*}
    \frac{1}{n}\log\int_{\Theta_m}\exp\bigl\{{\bf{H}}_{m,n}(\theta_m)\bigr\}\pi_m(\theta_m)d\theta_m
    =\frac{1}{n}{\bf{H}}_{m,n}(\hat{\theta}_{m,n})-\frac{q_m}{2n}\log n+O_p\Bigl(\frac{1}{n}\Bigr)
\end{align*}
as $n\longrightarrow\infty$.
\end{theorem}
\noindent
Therefore, we define the QBIC of SEM for jump-diffusion processes as
\begin{align*}
    {\bf{QBIC}}_{n}(m)=-2{\bf{H}}_{m,n}(\hat{\theta}_{m,n})+
    q_m\log n.
\end{align*}
Note that we select the model that minimizes \({\bf QBIC}_n(m)\).

We next examine the case in which the competing models include both correctly specified and misspecified parametric models. Note that  Model $m$ is misspecified if
\begin{align*}
    {\bf{\Sigma}}_0\neq {\bf{\Sigma}}_m(\theta_m)
\end{align*}
for all $\theta_m\in\Theta_m$. Let
\begin{align*}
    \mathcal{M}=\biggl\{m\in\{1,\ldots,M\}\ \Big|\ \mbox{There exists}\ \theta_{m,0}\in\Theta_{m}\ \mbox{such that}\ {\bf{\Sigma}}_0={\bf{\Sigma}}_{m}(\theta_{m,0}).\biggr\}
\end{align*}
and $\mathcal{M}^{c}=\{1,\ldots,M\}\setminus\mathcal{M}$. We define the optimal model $m^*$ as the model with the smallest number of parameters among correctly specified models, that is,
\begin{align*}
    m^*= {\rm{argmin}}_{m\in\mathcal{M}}\ {\rm{dim}}(\Theta_m),
\end{align*}
where it is supposed that the optimal model is unique. Furthermore, we define the optimal parameter $\bar{\theta}_{m}$ as follows:
\begin{align*}
    {\bf{H}}_{m}(\bar{\theta}_{m})=\sup_{\theta_m\in\Theta_m}{\bf{H}}_{m}(\theta_{m}).
\end{align*}
For $m\in\mathcal{M}$, it holds that $\bar{\theta}_{m}=\theta_{m,0}$. We also impose the following assumption:
\begin{enumerate}
    \vspace{3mm}
    \item[{\bf{[C2]}}] ${\bf{H}}_{m}(\theta_m)={\bf{H}}_{m}(\bar{\theta}_{m})\Longrightarrow \theta_m=\bar{\theta}_{m}$.
    \vspace{3mm}
\end{enumerate}
First, we have the following result for misspecified models.
\begin{theorem}\label{BICcons2} 
Under $\bf{[A1]}$-$\bf{[A4]}$, {\bf{[B]}} and {\bf{[C2]}}, as $n\longrightarrow\infty$,
\begin{align*}
    {\bf{P}}\Bigl({\bf{QBIC}}_{n}(m^*)<{\bf{QBIC}}_{n}(m)\Bigr)
    &\longrightarrow 1
\end{align*}
for all $m\in\mathcal{M}^c$.
\end{theorem}
\noindent
Theorem \ref{BICcons2} shows that QBIC does not asymptotically select misspecified models. We say that Model $i$ is nested in Model $j$ if $q_i<q_j$ and there exist a matrix $F_{i,j}\in\mathbb{R}^{q_j\times q_i}$ satisfying $F_{i,j}^{\top}F_{i,j}=\mathbb{I}_{q_i}$ and a vector $c_j\in\mathbb{R}^{q_j}$ such that
\begin{align*}
    {\bf{H}}_{i,n}(\theta_{i})={\bf{H}}_{j,n}(F_{i,j}\theta_{i}+c_j)
\end{align*}
for all $\theta_i\in\Theta_i$. Kusano and Uchida \cite{Kusano(jumpAIC)} proposed the quasi-Akaike information criterion of SEM for jump-diffusion processes as follows:
\begin{align*}
    {\bf{QAIC}}_{n}(m)=-2{\bf{H}}_{m,n}(\hat{\theta}_{m,n})+
    2q_m.
\end{align*}
As with QBIC, QAIC also does not asymptotically select misspecified models; see Proposition 2 in Kusano and Uchida \cite{Kusano(jumpAIC)}. However, as shown in the following theorem, QAIC does not have model-selection consistency.
\begin{theorem}\label{QAIC}
Under $\bf{[A1]}$-$\bf{[A4]}$, {\bf{[B]}} and {\bf{[C1]}}, if Model $m^*$ is nested in Model $m\in\mathcal{M}\setminus\{m^*\}$, then
\begin{align*}
    &\lim_{n\longrightarrow\infty} {\bf{P}}\Bigl({\bf{QAIC}}_n(m)<{\bf{QAIC}}_n(m^*)\Bigr)\\
    &\qquad\qquad\quad={\bf{P}}\Bigl(\chi^2_{q_m-q_{m^*}}>2(q_{m}-q_{m^*})\Bigr)>0.
\end{align*}
\end{theorem}
\noindent
On the other hand, the following theorem shows that QBIC has model-selection consistency.
\begin{theorem}\label{BICcons1} Under $\bf{[A1]}$-$\bf{[A4]}$, {\bf{[B]}} and {\bf{[C1]}}, if Model $m^*$ is nested in Model 
$m\in\mathcal{M}\setminus\{m^*\}$, then
\begin{align*}
    {\bf{P}}\Bigl({\bf{QBIC}}_{n}(m^*)<{\bf{QBIC}}_{n}(m)\Bigr)
    &\longrightarrow 1
\end{align*}
as $n\longrightarrow\infty$.
\end{theorem}
\section{Examples and Simulation results}\label{simulation}
\subsection{True model}
The $5$-dimensional observable process $\{X_{1,0,t}\}_{t\geq 0}$ is defined by the following true factor model:
\begin{align*}
    X_{1,0,t}&={\bf{\Lambda}}_{1,0}\xi_{0,t}+\delta_{0,t},
\end{align*}
where  $\{\xi_{0,t}\}_{t\geq 0}$ and $\{\delta_{0,t}\}_{t\geq 0}$ are $1$ and $5$-dimensional latent processes, respectively, and
\begin{align*}
    {\bf{\Lambda}}_{1,0}=\begin{pmatrix}
    1 & 0.2 & 0.4 & 0.1 & 0.7
    \end{pmatrix}^{\top}.
\end{align*}
The stochastic processes $\{\xi_{0,t}\}_{t\geq 0}$ and $\{\delta_{0,t}\}_{t\geq 0}$ are 
governed by the stochastic differential equations:
\begin{align*}
    d\xi_{0,t}&=-2(\xi_{0,t-}-1)dt+0.7d W_{1,t}+\int_{\mathbb{R}\setminus\{0\}}zp_1(dt,dz),\quad \xi_{0,0}=1
\end{align*}
and
\begin{align*}
    d\delta^{(1)}_{0,t}&=-3\delta^{(1)}_{0,t-}dt+0.9 dW^{(1)}_{2,t}
    +\int_{\mathbb{R}\setminus\{0\}}z p_{2,1}(dt,dz), \quad 
    \delta^{(1)}_{0,0}=0,\\
    d\delta^{(2)}_{0,t}&=-2\delta^{(2)}_{0,t-}dt+0.7dW^{(2)}_{2,t}
    +\int_{\mathbb{R}\setminus\{0\}}zp_{2,2}(dt,dz), \quad 
    \delta^{(2)}_{0,0}=0,\\
    d\delta^{(3)}_{0,t}&=-4\delta^{(3)}_{0,t-}dt+0.5dW^{(3)}_{2,t}
    +\int_{\mathbb{R}\setminus\{0\}}zp_{2,3}(dt,dz), \quad 
    \delta^{(3)}_{0,0}=0,\\
    d\delta^{(4)}_{0,t}&=-5\delta^{(4)}_{0,t-}dt+0.4dW^{(4)}_{2,t}
    +\int_{\mathbb{R}\setminus\{0\}}zp_{2,4}(dt,dz), \quad 
    \delta^{(4)}_{0,0}=0,\\
    d\delta^{(5)}_{0,t}&=-2\delta^{(5)}_{0,t-}dt+0.8dW^{(5)}_{2,t}
    +\int_{\mathbb{R}\setminus\{0\}}zp_{2,5}(dt,dz), \quad 
    \delta^{(5)}_{0,0}=0,
\end{align*}
where $\{W_{1,t}\}_{t\geq 0}$ and $\{W_{2,t}\}_{t\geq 0}$ are $1$ and $5$-dimensional standard Wiener processes, and $p_1(dt,dz)$ and $p_{2,i}(dt,dz)$ for $i=1,2,3,4,5$ are Poisson random measures on $\mathbb{R}^+\times \mathbb{R}\setminus\{0\}$ with the jump densities $f_{1}(z)=2g(z|0,5)$ and $f_{2,1}(z)=g(z|0,5)$, $f_{2,2}(z)=g(z|0,4)$, $f_{2,3}(z)=g(z|0,6)$, $f_{2,4}(z)=g(z|0,5)$, $f_{2,5}(z)=g(z|0,4)$, respectively. Here, $g(z|\mu,\sigma^2)$ is the probability density function of the normal distribution with mean $\mu\in\mathbb{R}$ and variance $\sigma^2>0$, i.e.,
\begin{align*}
    g(z|\mu,\sigma^2)=\frac{1}{\sqrt{2\pi\sigma^2}}\exp{\biggl\{-\frac{(z-\mu)^2}{2\sigma^2}\biggr\}}.
\end{align*}
The $10$-dimensional observable process $\{X_{2,0,t}\}_{t\geq 0}$ is given by the following factor model:
\begin{align*}
    X_{2,0,t}&={\bf{\Lambda}}_{2,0}\eta_{0,t}+\varepsilon_{0,t},
\end{align*}
where $\{\eta_{0,t}\}_{t\geq 0}$ and $\{\varepsilon_{0,t}\}_{t\geq 0}$ are $2$ and $10$-dimensional latent processes, respectively, and
\begin{align*}
    {\bf{\Lambda}}_{2,0}=\begin{pmatrix}
    1 & 0.2 & 0.9 & 1.2 & 0.3 & 0 & 0 & 0 & 0 & 0\\
    0 & 0 & 0 & 0 & 0 & 1 & 0.5 & 0.6 & 0.4 & 0.7
    \end{pmatrix}^{\top}.
\end{align*}
The stochastic process $\{\varepsilon_{0,t}\}_{t\geq 0}$ is 
defined by the stochastic differential equations:
\begin{align*}
    d\varepsilon^{(1)}_{0,t}&=-2\varepsilon^{(1)}_{0,t-}dt+0.4dW^{(1)}_{3,t}
    +\int_{\mathbb{R}\setminus\{0\}}z p_{3,1}(dt,dz),\quad 
    \varepsilon^{(1)}_{0,0}=0,\\
    d\varepsilon^{(2)}_{0,t}&=-3\varepsilon^{(2)}_{0,t-}dt+0.9dW^{(2)}_{3,t}
    +\int_{\mathbb{R}\setminus\{0\}}zp_{3,2}(dt,dz),\quad 
    \varepsilon^{(2)}_{0,0}=0,\\
    d\varepsilon^{(3)}_{0,t}&=-2\varepsilon^{(3)}_{0,t-}dt+0.3dW^{(3)}_{3,t}
    +\int_{\mathbb{R}\setminus\{0\}}zp_{3,3}(dt,dz),\quad 
    \varepsilon^{(3)}_{0,0}=0,\\
    d\varepsilon^{(4)}_{0,t}&=-5\varepsilon^{(4)}_{0,t-}dt+0.6
    dW^{(4)}_{3,t}+\int_{\mathbb{R}\setminus\{0\}}zp_{3,4}(dt,dz),\quad 
    \varepsilon^{(4)}_{0,0}=0,\\
    d\varepsilon^{(5)}_{0,t}&=-4\varepsilon^{(5)}_{0,t-}dt
    +0.4dW^{(5)}_{3,t}+\int_{\mathbb{R}\setminus\{0\}}zp_{3,5}(dt,dz),\quad 
    \varepsilon^{(5)}_{0,0}=0,\\
    d\varepsilon^{(6)}_{0,t}&=-2\varepsilon^{(6)}_{0,t-}dt
    +0.5dW^{(6)}_{3,t}+\int_{\mathbb{R}\setminus\{0\}}zp_{3,6}(dt,dz),\quad 
    \varepsilon^{(6)}_{0,0}=0,\\
    d\varepsilon^{(7)}_{0,t}&=-3\varepsilon^{(7)}_{0,t-}dt
    +0.8dW^{(7)}_{3,t}+\int_{\mathbb{R}\setminus\{0\}}zp_{3,7}(dt,dz),\quad 
    \varepsilon^{(7)}_{0,0}=0,\\
    d\varepsilon^{(8)}_{0,t}&=-2\varepsilon^{(8)}_{0,t-}dt
    +0.6dW^{(8)}_{3,t}+\int_{\mathbb{R}\setminus\{0\}}zp_{3,8}(dt,dz),\quad 
    \varepsilon^{(8)}_{0,0}=0,\\   
    d\varepsilon^{(9)}_{0,t}&=-5\varepsilon^{(9)}_{0,t-}dt
    +0.7dW^{(9)}_{3,t}+\int_{\mathbb{R}\setminus\{0\}}zp_{3,9}(dt,dz),\quad 
    \varepsilon^{(9)}_{0,0}=0,\\
    d\varepsilon^{(10)}_{0,t}&=-4\varepsilon^{(10)}_{0,t-}dt
    +0.3dW^{(10)}_{3,t}+\int_{\mathbb{R}\setminus\{0\}}zp_{3,10}(dt,dz),\quad 
    \varepsilon^{(10)}_{0,0}=0,
\end{align*}
where $\{W_{3,t}\}_{t\geq 0}$ is a $10$-dimensional standard Wiener process, and 
$p_{3,i}(dt,dz)$ for $i=1,\ldots,10$ are Poisson random measures on $\mathbb{R}^+\times \mathbb{R}\setminus\{0\}$ with the jump densities $f_{3,1}(z)=g(z|0,5)$, $f_{3,2}(z)=g(z|0,4)$, $f_{3,3}(z)=g(z|0,4)$, $f_{3,4}(z)=g(z|0,5)$, $f_{3,5}(z)=g(z|0,6)$, $f_{3,6}(z)=g(z|0,4)$, $f_{3,7}(z)=g(z|0,6)$, $f_{3,8}(z)=g(z|0,5)$, $f_{3,9}(z)=g(z|0,6)$  and $f_{3,10}(z)=g(z|0,5)$, respectively. Furthermore, we assume that the relationship between 
$\{\xi_{0,t}\}_{t\geq 0}$ and  $\{\eta_{0,t}\}_{t\geq 0}$ 
is defined by
\begin{align*}
    \eta_{0,t}={\bf{\Gamma}}_0\xi_{0,t}+\zeta_{0,t},
\end{align*}
where $\{\zeta_{0,t}\}_{t\geq 0}$ is a $2$-dimensional latent process and
\begin{align*}
    {\bf{\Gamma}}_0=\begin{pmatrix}
    0.7 & -0.5
    \end{pmatrix}^{\top}.
\end{align*}
The stochastic process $\{\zeta_{0,t}\}_{t\geq 0}$ is assumed to satisfy the following stochastic differential equation:
\begin{align*}
    d\zeta^{(1)}_{0,t}&=-5\zeta^{(1)}_{0,t-}dt+0.5dW^{(1)}_{4,t}
    +\int_{\mathbb{R}\setminus\{0\}}z p_{4,1}(dt,dz), \quad 
    \zeta^{(1)}_{0,0}=0,\\
    d\zeta^{(2)}_{0,t}&=-2\zeta^{(2)}_{0,t-}dt+0.8dW^{(2)}_{4,t}
    +\int_{\mathbb{R}\setminus\{0\}}zp_{4,2}(dt,dz), \quad 
    \zeta^{(2)}_{0,0}=0,
\end{align*}
where $\{W_{4,t}\}_{t\geq 0}$ is a $2$-dimensional standard Wiener process, and $p_{4,1}(dt,dz)$ and $p_{4,2}(dt,dz)$ are Poisson random measures on $\mathbb{R}^+\times \mathbb{R}\setminus\{0\}$ with the jump densities $f_{4,1}(z)=g(z|0,6)$ and $f_{4,2}(z)=g(z|0,5)$, respectively. The path diagram of the true model is shown in Figure \ref{modeltrue}.
\begin{figure}[h]
    \includegraphics[width=0.9\columnwidth]{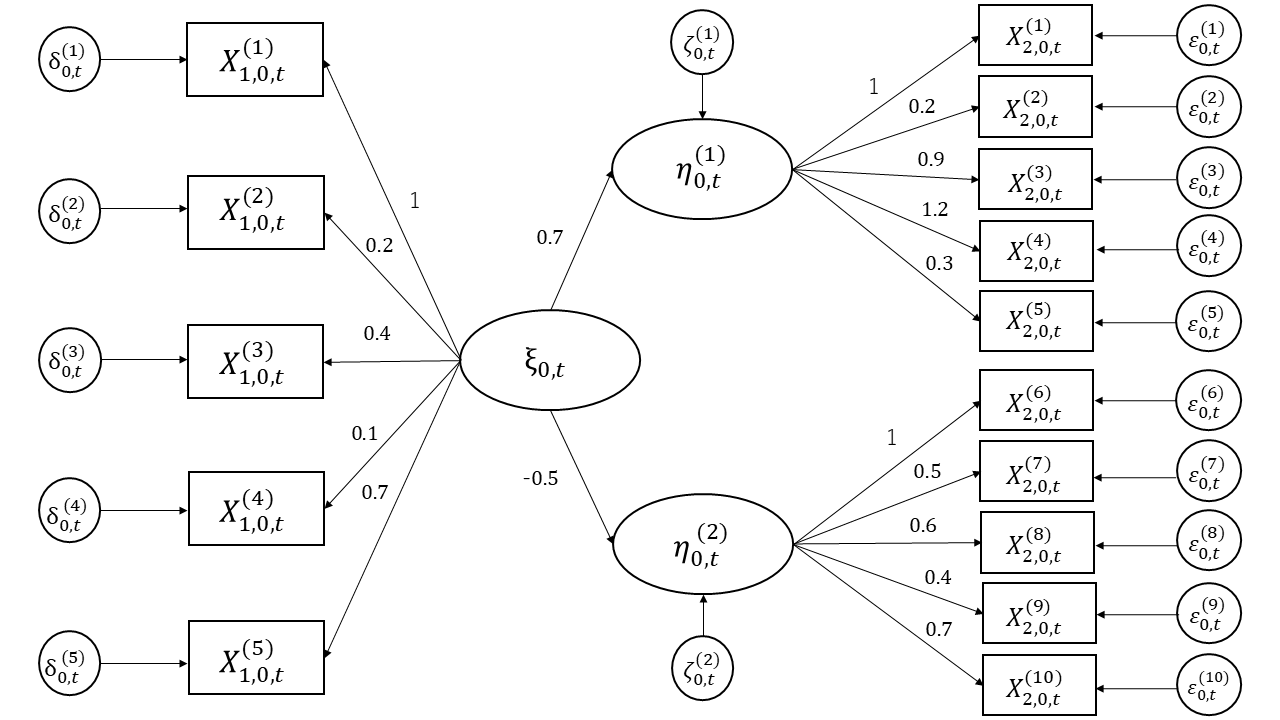}
    \caption{The path diagram of the true model at time $t$.} \label{modeltrue}
\end{figure}
\subsection{Model 1}
Set $p_1=5$, $p_2=10$, $k_1=1$, $k_2=2$ and $q_1=32$. Assume that
\begin{align*}
    &\qquad\qquad\qquad\quad
    {\bf{\Lambda}}_{1,1}^{\theta}=\begin{pmatrix}
    1 & \theta_1^{(1)} & \theta_1^{(2)} & \theta_1^{(3)} & \theta_1^{(4)}
    \end{pmatrix}^{\top},\\
    &{\bf{\Lambda}}_{2,1}^{\theta}=\begin{pmatrix}
    1 & \theta_1^{(5)} & \theta_1^{(6)} & \theta_1^{(7)} & \theta_1^{(8)} & 0 & 0 & 0 & 0 & 0\\
    0 & 0 & 0 & 0 & 0 & 1 & \theta_1^{(9)} & \theta_1^{(10)} & \theta_1^{(11)} & \theta_1^{(12)}
    \end{pmatrix}^{\top}
\end{align*}
and
\begin{align*}
    {\bf{\Gamma}}_1^{\theta}=\begin{pmatrix}
    \theta_1^{(13)},\theta_1^{(14)}
    \end{pmatrix}^{\top},\quad {\bf{\Psi}}_1^{\theta}=\mathbb{I}_2,
\end{align*}
where $\theta_1^{(i)}$ for $i=1,\ldots,14$ are not zero. In addition, we suppose that
\begin{align*}
    &\qquad
    {\bf{\Sigma}}_{\xi\xi,1}^{\theta}=\theta_1^{(15)},\quad 
    {\bf{\Sigma}}_{\delta\delta,1}^{\theta}=\Diag\bigl( \theta_1^{(16)}, \theta_1^{(17)},\theta_1^{(18)},\theta_1^{(19)},
    \theta_1^{(20)}\bigr)^{\top},\\
    &{\bf{\Sigma}}_{\varepsilon\varepsilon,1}^{\theta}=\Diag\bigl( 
    \theta_1^{(21)}, \theta_1^{(22)},\theta_1^{(23)},\theta_1^{(24)},
    \theta_1^{(25)}, \theta_1^{(26)},\theta_1^{(27)},\theta_1^{(28)},
    \theta_1^{(29)}, \theta_1^{(30)}\bigr)^{\top}
\end{align*}
and
\begin{align*}
    {\bf{\Sigma}}_{\zeta\zeta,1}^{\theta}
    =\Diag\bigl(\theta_1^{(31)}, \theta_1^{(32)}\bigr)^{\top},
\end{align*}
where $\theta_1^{(i)}$ for $i=15,\ldots,32$ are positive. This model is correctly specified since 
\begin{align*}
    {\bf{\Sigma}}_0={\bf{\Sigma}}_1(\theta_{1,0}),
\end{align*}
where 
\begin{align*}
    \theta_{1,0}&=\bigl(0.2,0.4,0.1,0.7,0.2,0.9,1.2,0.3,0.5,0.6,0.4,0.7,0.7,-0.5,0.49,0.81,0.49, \\
    &\qquad\quad 
    0.25,0.16 ,0.64,0.16,0.81,0.09,0.36,0.16,0.25,0.64,0.36
    ,0.49,0.09,0.25,0.64\bigr)^{\top}.
\end{align*}
In an analogous manner to the Appendix of Kusano and Uchida \cite{Kusano(BIC)}, it follows that
\begin{align*}
    {\bf{\Sigma}}_1(\theta_{1})={\bf{\Sigma}}_1(\theta_{1,0})\Longrightarrow\theta_1=\theta_{1,0}.
\end{align*}
In addition, since $\rank\Delta_{1,0}=32$, ${\bf{[C1]}}$ is satisfied.
The path diagram of the Model 1 is shown in Figure \ref{model1}.
\begin{figure}[h]
    \includegraphics[width=0.9\columnwidth]{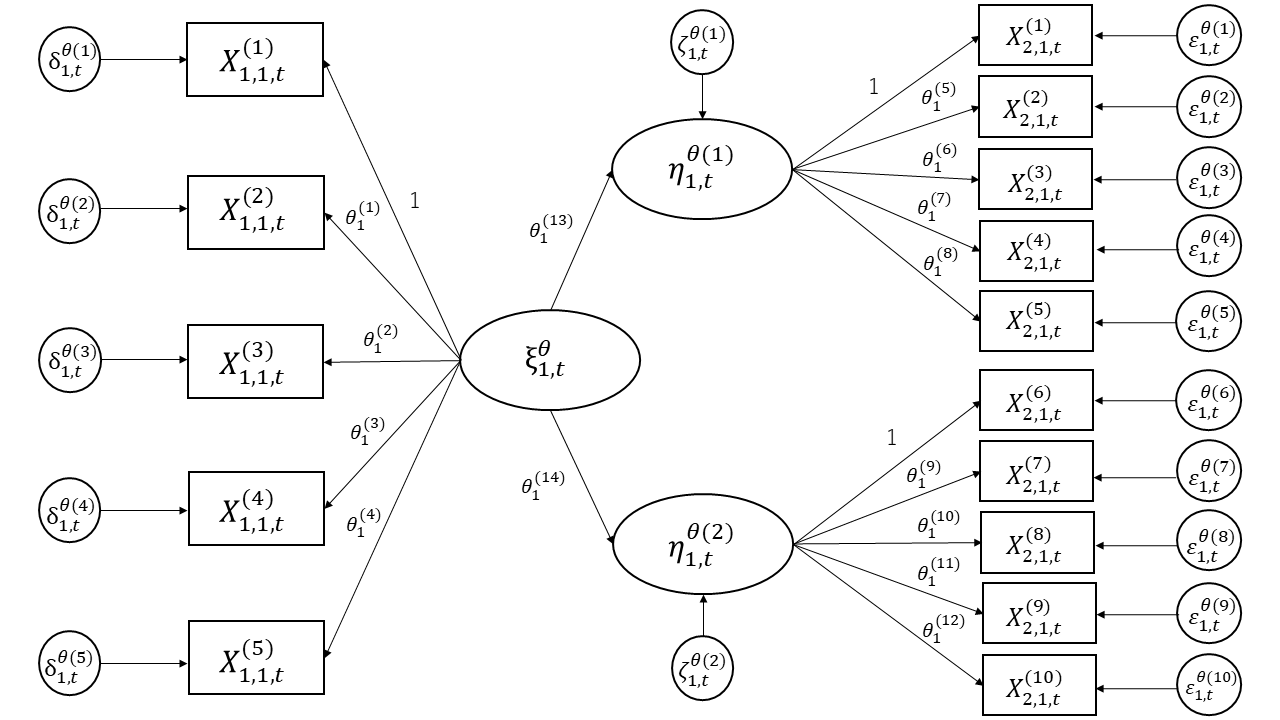}
    \caption{The path diagram of Model 1 at time $t$.} \label{model1}
\end{figure}
\subsection{Model 2}
Set $p_1=5$, $p_2=10$, $k_1=1$, $k_2=2$ and $q_2=33$. Supposed that
\begin{align*}
    &\qquad\qquad\qquad
    {\bf{\Lambda}}_{1,2}^{\theta}=\begin{pmatrix}
    1 & \theta_2^{(1)} & \theta_2^{(2)} & \theta_2^{(3)} & \theta_2^{(4)}
    \end{pmatrix}^{\top},\\
    &{\bf{\Lambda}}_{2,2}^{\theta}=\begin{pmatrix}
    1 & \theta_2^{(5)} & \theta_2^{(6)} & \theta_2^{(7)} & \theta_2^{(8)} & 0 & 0 & 0 & 0 & 0\\
    0 & 0 & 0 & 0 & \theta_2^{(9)} & 1 & \theta_2^{(10)} & \theta_2^{(11)} & \theta_2^{(12)} & \theta_2^{(13)}
    \end{pmatrix}^{\top}
\end{align*}
and
\begin{align*}
    {\bf{\Gamma}}_2^{\theta}=\begin{pmatrix}
    \theta_2^{(14)},\theta_2^{(15)}
    \end{pmatrix}^{\top},\quad {\bf{\Psi}}_2^{\theta}=\mathbb{I}_2,
\end{align*}
where $\theta_2^{(i)}$ for $i=1,\ldots, 8$ and $10,\ldots,15$ are not zero. Furthermore, we assume that
\begin{align*}
    &\qquad
    {\bf{\Sigma}}_{\xi\xi,2}^{\theta}=\theta_2^{(16)},\quad 
    {\bf{\Sigma}}_{\delta\delta,2}^{\theta}=\Diag\bigl( \theta_2^{(17)},\theta_2^{(18)},\theta_2^{(19)},
    \theta_2^{(20)},
    \theta_2^{(21)}\bigr)^{\top},\\
    &{\bf{\Sigma}}_{\varepsilon\varepsilon,2}^{\theta}=\Diag\bigl( 
    \theta_2^{(22)},\theta_2^{(23)},\theta_2^{(24)},
    \theta_2^{(25)}, \theta_2^{(26)},\theta_2^{(27)},\theta_2^{(28)},
    \theta_2^{(29)}, \theta_2^{(30)}, \theta_2^{(31)}\bigr)^{\top}
\end{align*}
and
\begin{align*}
    {\bf{\Sigma}}_{\zeta\zeta,2}^{\theta}
    =\Diag\bigl(\theta_2^{(32)}, \theta_2^{(33)}\bigr)^{\top},
\end{align*}
where $\theta_2^{(i)}$ for $i=16,\ldots,33$ are positive. This model is correctly specified since 
\begin{align*}
    {\bf{\Sigma}}_0={\bf{\Sigma}}_2(\theta_{2,0}),
\end{align*}
where 
\begin{align*}
    \theta_{2,0}&=\bigl(0.2,0.4,0.1,0.7,0.2,0.9,1.2,0.3,0,0.5, 0.6,0.4,0.7,0.7,-0.5,0.49,0.81,0.49, \\
    &\qquad\quad 
    0.25,0.16 ,0.64,0.16,0.81,0.09,0.36,0.16,0.25,0.64,0.36
    ,0.49,0.09,0.25,0.64\bigr)^{\top}.
\end{align*}
In a similar way to the Appendix of Kusano and Uchida \cite{Kusano(jumpAIC)}, we have
\begin{align*}
    {\bf{\Sigma}}_2(\theta_{2})={\bf{\Sigma}}_2(\theta_{2,0})\Longrightarrow\theta_2=\theta_{2,0}. 
\end{align*}
Hence, $\rank\Delta_{2,0}=33$ implies ${\bf{[C1]}}$. 
The path diagram of the Model 2 is shown in Figure \ref{model2}.
\begin{figure}[h]
    \includegraphics[width=0.9\columnwidth]{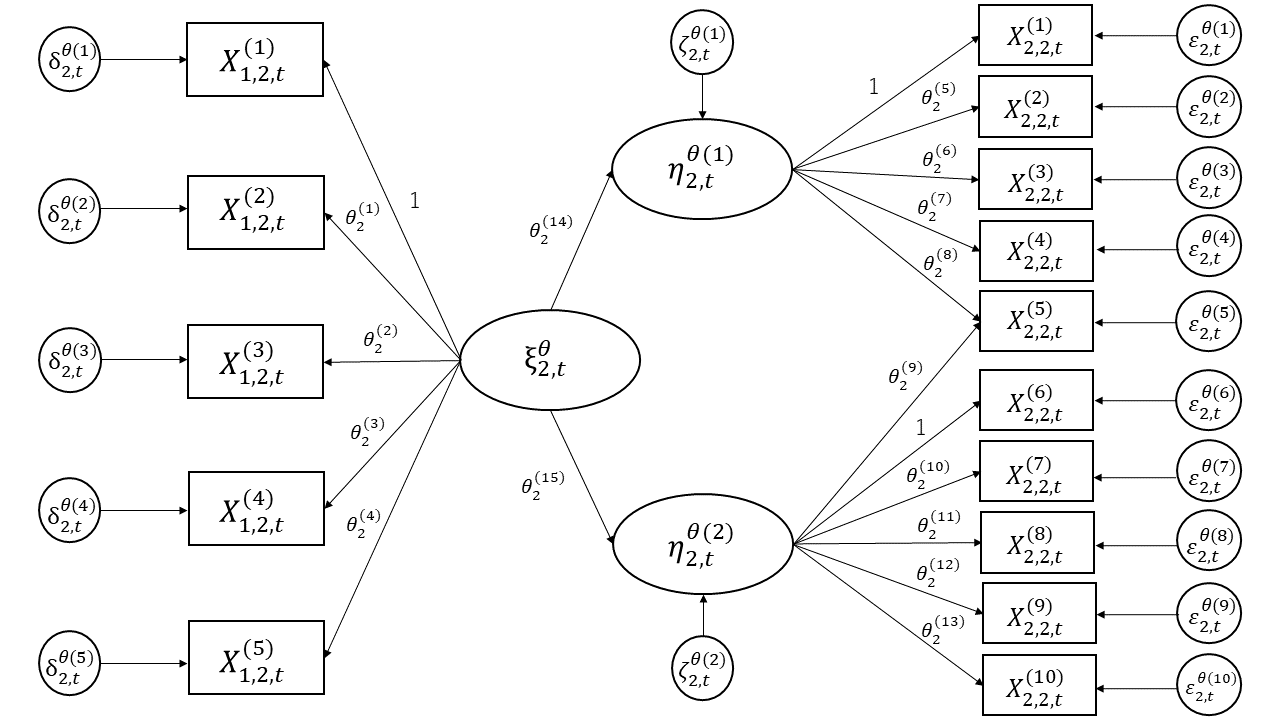}
    \caption{The path diagram of Model 2 at time $t$.} \label{model2}
\end{figure}
\subsection{Model 3}
Set $p_1=5$, $p_2=10$, $k_1=1$, $k_2=1$ and $q_3=31$. Assume that
\begin{align*}
    {\bf{\Lambda}}_{1,3}^{\theta}=\begin{pmatrix}
    1 & \theta_3^{(1)} & \theta_3^{(2)} & \theta_3^{(3)}
    & \theta_3^{(4)}
    \end{pmatrix}^{\top}
\end{align*}
and
\begin{align*}
    {\bf{\Lambda}}_{2,3}^{\theta}=\begin{pmatrix}
    1 & \theta_3^{(5)} & \theta_3^{(6)} &\theta_3^{(7)} & \theta_3^{(8)} & \theta_3^{(9)} & \theta_3^{(10)} & \theta_3^{(11)} & \theta_3^{(12)} & \theta_3^{(13)}\end{pmatrix}^{\top},\quad {\bf{\Gamma}}_3^{\theta}=
    \theta_3^{(14)},
\end{align*}
where $\theta_3^{(i)}$ for $i=1,\ldots,14$ are not zero. Suppose that
\begin{align*}
    {\bf{\Sigma}}_{\xi\xi,3}^{\theta}=\theta_3^{(15)},\quad 
    {\bf{\Sigma}}_{\delta\delta,3}^{\theta}=\Diag\bigl(\theta_3^{(16)},
    \theta_3^{(17)}, \theta_3^{(18)},\theta_3^{(19)},\theta_3^{(20)}  \bigr)^{\top}
\end{align*}
and
\begin{align*}
    {\bf{\Sigma}}_{\varepsilon\varepsilon,3}^{\theta}=\Diag\bigl(\theta_3^{(21)}, \theta_3^{(22)},\theta_3^{(23)},\theta_3^{(24)}, \theta_3^{(25)}, \theta_3^{(26)}, \theta_3^{(27)},\theta_3^{(28)}, \theta_3^{(29)}, \theta_3^{(30)}\bigr)^{\top},\quad {\bf{\Sigma}}_{\zeta\zeta,3}^{\theta}=\theta_3^{(31)},
\end{align*}
where $\theta_3^{(i)}$ for $i=15,\ldots,31$ are positive. Since
\begin{align*}
    {\bf{\Sigma}}_0\neq{\bf{\Sigma}}_3(\theta_{3})
\end{align*}
for all $\theta_3\in\Theta_3$, Model $3$ is a misspecified model. The path diagram of the Model 3 is shown in Figure \ref{model3}.
\begin{figure}[h]
    \includegraphics[width=0.9\columnwidth]{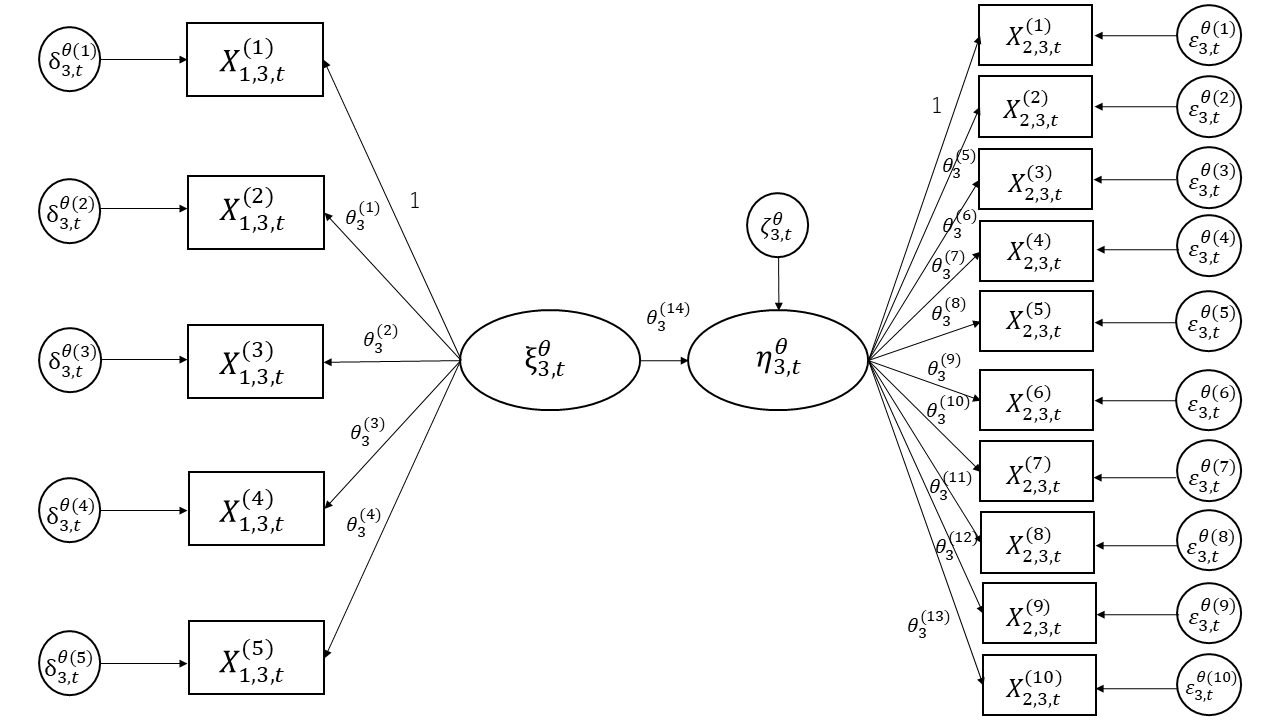}
    \caption{The path diagram of Model 3 at time $t$.} \label{model3}
\end{figure}
\subsection{Simulation results}
The simulation study was carried out using the optim() function in R with the BFGS algorithm. The optimization was initialized at the true parameter value. We conducted 10000 simulations with $T=1$, $D=10$, and $\rho=0.4$. The sample sizes were set to $n=5\times 10^4, 10^5, 5\times 10^5, 10^6$. Tables \ref{table1} and \ref{table2} show the number of times each model was selected by QAIC and QBIC. Whereas QBIC tends to select the optimal model (Model 1) as the sample size increases, QAIC often selects the overfitted model (Model 2) even when the sample size is large. This result implies that Theorems \ref{QAIC} and \ref{BICcons1} hold. Furthermore, neither QAIC nor QBIC chose the misspecified model (Model 3). This implies that Proposition 2 in Kusano and Uchida \cite{Kusano(jumpAIC)} and Theorem \ref{BICcons2} hold.
\begin{table}[h]
\centering
\begin{tabular}{cccccc}
    &  $n=5\times 10^4$ & $n=10^5$ & $n=5\times10^5$ & $n=10^6$\\\hline
    Model 1$^*$ & 8248 & 8295& 8380 & 8375 \\
    Model 2 & 1752 & 1705 & 1620 & 1625 \\
    Model 3 & 0 & 0 & 0 & 0 
\end{tabular}
    \caption{The number of times each model was selected by QAIC.}\label{table1}
\end{table}
\begin{table}
\begin{tabular}{cccccc}
    & $n=5\times 10^4$ & $n=10^5$ & $n=5\times10^5$ & $n=10^6$\\\hline
    Model 1$^*$ & 9971 & 9983 & 9995 & 9999\\
    Model 2 & 29 & 17 & 5 & 1 \\
    Model 3 & 0 & 0 & 0 & 0 
\end{tabular}
    \caption{The number of times each model was selected by QBIC.}\label{table2}
\end{table}
\section{Proofs}\label{proofs}
In this section, we may omit the model index $m$ for simplicity. Let 
\begin{align*}
    {\bf{\Gamma}}_n(\theta)=-\frac{1}{n}\partial^2_{\theta}{\bf{H}}_n(\theta),\quad {\bf{\Gamma}}(\theta)=-\partial^2_{\theta}{\bf{H}}(\theta)
\end{align*}
and
\begin{align*}
    {\bf{Y}}_n(\theta)=\frac{1}{n}{\bf{H}}_n(\theta)-\frac{1}{n}{\bf{H}}_n(\theta_0)
\end{align*}
for $\theta\in\Theta$. Set 
\begin{align*}
    \hat{\bf{Z}}_n(u)=\exp\Bigl\{{\bf{H}}_n(\hat{\theta}_n+n^{-1/2}u)-{\bf{H}}_n(\hat{\theta}_n)\Bigr\}
\end{align*}
for $u\in\hat{\mathbb{U}}_n$, where 
\begin{align*}
    \hat{\mathbb{U}}_n=\Bigl\{u\in\mathbb{R}^q\, \big|\, \hat{\theta}_n+n^{-1/2}u\in\Theta\Bigr\}.
\end{align*}
Furthermore, we define
\begin{align*}
    \tilde{\bf{Z}}_n(u)=\begin{cases}
    \hat{\bf{Z}}_n(u) & (u\in\hat{\mathbb{U}}_n\cap A_n),\\
    1 & (u\in(\hat{\mathbb{U}}_n\cap A_n)^c)
\end{cases},\quad {\bf{Z}}(u)=\exp\Bigl\{-\frac{1}{2}u^{\top}{\bf{\Gamma}}(\theta_0)u\Bigr\}
\end{align*}
for $u\in\mathbb{R}^q$, where
\begin{align*}
    A_n=\Bigl\{u\in\mathbb{R}^q\, \big|\, n^{-1/2}|u|\leq n^{-1/8}\Bigr\}.
\end{align*}
The proof of Theorem \ref{Hexpansion} is based on the approach used in the proof of Theorem 3 of Eguchi and Masuda \cite{Eguchi(2024)}.
\begin{lemma}\label{thetahat}
Under $\bf{[A1]}$-$\bf{[A4]}$ and $\bf{[C1]}$,
\begin{align*}
    \hat{\theta}_n\stackrel{p}{\longrightarrow}\theta_0
\end{align*}
and
\begin{align*}
    \sqrt{n}(\hat{\theta}_n-\theta_0)\stackrel{d}{\longrightarrow}N_q\bigl(0,{\bf{\Gamma}}(\theta_0)^{-1}\bigr)
\end{align*}
as $n\longrightarrow\infty$.
\end{lemma}
\begin{proof}
In a similar way to Theorem 2 in Kusano and Uchida \cite{Kusano(jump)}, the result can be proven.
\end{proof}
\begin{lemma}\label{minlambda}
Under $\bf{[A1]}$-$\bf{[A4]}$ and $\bf{[C1]}$, there exists $\lambda_0>0$ such that
\begin{align*}
    {\bf{P}}\Bigl(\lambda_{min}\, {\bf{\Gamma}}(\hat{\theta}_n)\geq 4\lambda_0\Bigr)\longrightarrow 1
\end{align*}
as $n\longrightarrow\infty$.
\end{lemma}
\begin{proof}
Since $\bf{[C1]}$ implies that ${\bf{\Gamma}}(\theta_0)$ is positive definite, there exists $\lambda_0>0$ such that
\begin{align}
    \lambda_{min}\, {\bf{\Gamma}}(\theta_0)\geq 5\lambda_0.\label{lambdamin5}
\end{align}
By the continuous mapping theorem and Lemma \ref{thetahat}, it holds that
\begin{align}
    \lambda_{min}\, {\bf{\Gamma}}(\hat{\theta}_n)\stackrel{p}{\longrightarrow}\lambda_{min}\, {\bf{\Gamma}}(\theta_0).\label{lambdaminprob}
\end{align}
On 
\begin{align*}
    \Bigl\{\lambda_{min}\, {\bf{\Gamma}}(\hat{\theta}_n)<4\lambda_0\Bigr\}\cap\Bigl\{\bigl|\lambda_{min}\, {\bf{\Gamma}}\, (\hat{\theta}_n)-\lambda_{min}\, {\bf{\Gamma}}(\theta_0)\bigr|<\lambda_0\Bigr\},
\end{align*}
we see
\begin{align*}
    -\lambda_0+\lambda_{min}\, {\bf{\Gamma}}(\theta_0)
    <\lambda_{min}\, {\bf{\Gamma}}(\hat{\theta}_n)<4\lambda_0,
\end{align*} 
which yields
\begin{align*}
    \lambda_{min}\, {\bf{\Gamma}}(\theta_0)<5\lambda_0.
\end{align*}
Therefore, it follows from (\ref{lambdamin5}) and (\ref{lambdaminprob}) that
\begin{align*}
    &\quad\ {\bf{P}}\Bigl(\lambda_{min}\, {\bf{\Gamma}}(\hat{\theta}_n)<4\lambda_0\Bigr)\\
    &={\bf{P}}\Bigl(\Bigl\{\lambda_{min}\, {\bf{\Gamma}}(\hat{\theta}_n)<4\lambda_0\Bigr\}\cap\Bigl\{\bigl|\lambda_{min}\, {\bf{\Gamma}}\, (\hat{\theta}_n)-\lambda_{min}\, {\bf{\Gamma}}(\theta_0)\bigr|<\lambda_0\Bigr\}\Bigr)\\
    &\qquad+{\bf{P}}\Bigl(\Bigl\{\lambda_{min}\, {\bf{\Gamma}}(\hat{\theta}_n)<4\lambda_0\Bigr\}\cap\Bigl\{\bigl|\lambda_{min}\, {\bf{\Gamma}}\, (\hat{\theta}_n)-\lambda_{min}\, {\bf{\Gamma}}(\theta_0)\bigr|\geq\lambda_0\Bigr\}\Bigr)\\
    &\leq{\bf{P}}\Bigl(\lambda_{min}\, {\bf{\Gamma}}(\theta_0)<5\lambda_0\Bigr)+{\bf{P}}\Bigl(\bigl|\lambda_{min}\, {\bf{\Gamma}}\, (\hat{\theta}_n)-\lambda_{min}\, {\bf{\Gamma}}(\theta_0)\bigr|\geq\lambda_0\Bigr)\\
    &=0+{\bf{P}}\Bigl(\bigl|\lambda_{min}\, {\bf{\Gamma}}\, (\hat{\theta}_n)-\lambda_{min}\, {\bf{\Gamma}}(\theta_0)\bigr|\geq\lambda_0\Bigr)\longrightarrow 0
\end{align*}
as $n\longrightarrow\infty$, which deduces 
\begin{align*}
    {\bf{P}}\Bigl(\lambda_{min}\, {\bf{\Gamma}}(\hat{\theta}_n)\geq 4\lambda_0\Bigr)\longrightarrow 1
\end{align*}
as $n\longrightarrow\infty$. This completes the proof.
\end{proof}
\begin{lemma}\label{problemma}
Under $\bf{[A1]}$-$\bf{[A4]}$ and $\bf{[C1]}$, 
\begin{align}
    &\qquad\sup_{\theta\in\Theta}\bigl|{\bf{\Gamma}}_n(\theta)-{\bf{\Gamma}}(\theta)\bigr|\stackrel{p}{\longrightarrow}0,\label{gammaprob}\\
    &\sqrt{n}\sup_{\theta\in\Theta}\bigl|{\bf{Y}}_n(\theta)-{\bf{Y}}(\theta)\bigr|=O_p(1) \label{Yprob}
\end{align}
and
\begin{align}
    \frac{1}{n}\sup_{\theta\in\Theta}\bigl|\partial^3_{\theta}{\bf{H}}_n(\theta)\bigr|=O_p(1) \label{H3prob}
\end{align}
as $n\longrightarrow\infty$.
\end{lemma}
\begin{proof}
In an analogous manner to Proposition 7 in Kusano and Uchida \cite{Kusano(jump)}, we can prove (\ref{gammaprob}). Lemma 2 in Kusano and Uchida \cite{Kusano(jumpAIC)} implies (\ref{H3prob}). Therefore, we only prove (\ref{Yprob}). Set
\begin{align*}
    N_n=\sum_{i=1}^n{\bf{1}}_{\{|\Delta_{i}^n X|\leq Dh_n^{\rho}\}},\quad \check{\bf{\Sigma}}_n=\frac{1}{nh_n}\sum_{i=1}^n(\Delta_i^n X)(\Delta_i^n X)^{\top}{\bf{1}}_{\{|\Delta_i^n X|\leq Dh_n^{\rho}\}}.
\end{align*}
Since 
\begin{align*}
    \frac{1}{n}{\bf{H}}_n(\theta)&=-\frac{1}{2}\tr\Bigl\{{\bf{\Sigma}}(\theta)^{-1}\check{\bf{\Sigma}}_n\Bigr\}-\frac{N_n}{2n}\log\det {\bf{\Sigma}}(\theta)
\end{align*}
and
\begin{align*}
    {\bf{H}}(\theta)=-\frac{1}{2}\tr\Bigl\{{\bf{\Sigma}}(\theta)^{-1}{\bf{\Sigma}}(\theta_0)\Bigr\}-\frac{1}{2}\log\det {\bf{\Sigma}}(\theta),
\end{align*}
it follows that
\begin{align*}
    &\quad\ \frac{1}{n}{\bf{H}}_n(\theta)-{\bf{H}}(\theta)\\
    &=-\frac{1}{2}
    \tr\Bigl\{{\bf{\Sigma}}(\theta)^{-1}\bigl(\check{\bf{\Sigma}}_n-{\bf{\Sigma}}(\theta_0)\bigr)\Bigr\}-\frac{1}{2}\log\det {\bf{\Sigma}}(\theta)\Bigl(\frac{N_n}{n}-1\Bigr).
\end{align*}
Consequently, we have
\begin{align*}
    {\bf{Y}}_n(\theta)-{\bf{Y}}(\theta)
    &=\Bigl(\frac{1}{n}{\bf{H}}_n(\theta)-{\bf{H}}(\theta)\Bigr)-\Bigl(\frac{1}{n}{\bf{H}}_n(\theta_0)-{\bf{H}}(\theta_0)\Bigr)\\
    &=-\frac{1}{2}
    \tr\Bigl\{{\bf{\Sigma}}(\theta)^{-1}\bigl(\check{\bf{\Sigma}}_n-{\bf{\Sigma}}(\theta_0)\bigr)\Bigr\}\\
    &\quad-\frac{1}{2}\log\det {\bf{\Sigma}}(\theta)\Bigl(\frac{N_n}{n}-1\Bigr)\\
    &\quad+\frac{1}{2}
    \tr\Bigl\{{\bf{\Sigma}}(\theta_0)^{-1}\bigl(\check{\bf{\Sigma}}_n-{\bf{\Sigma}}(\theta_0)\bigr)\Bigr\}+\frac{1}{2}\log\det {\bf{\Sigma}}(\theta_0)\Bigl(\frac{N_n}{n}-1\Bigr).
\end{align*}
By the proofs of Proposition 1 and Theorem 1 in Kusano and Uchida \cite{Kusano(jump)}, it is shown that
\begin{align*}
    \sqrt{n}\Bigl(\frac{N_n}{n}-1\Bigr)=o_p(1),\quad \sqrt{n}\bigl(\check{\bf{\Sigma}}_n-{\bf{\Sigma}}(\theta_0)\bigr)=O_p(1),
\end{align*}
so that 
\begin{align*}
    \sqrt{n}\sup_{\theta\in\Theta}\bigl|{\bf{Y}}_n(\theta)-{\bf{Y}}(\theta)\bigr|
    &\leq \frac{1}{2}\sup_{\theta\in\Theta}\bigl|{\bf{\Sigma}}(\theta)^{-1}\bigr|
    \Bigl|\sqrt{n}\bigl(\check{\bf{\Sigma}}_n-{\bf{\Sigma}}(\theta_0)\bigr)\Bigr|\\
    &\quad+\frac{1}{2}\sup_{\theta\in\Theta}\bigl|\log\det {\bf{\Sigma}}(\theta)\bigr|\Bigl|\sqrt{n}\Bigl(\frac{N_n}{n}-1\Bigr)\Bigr|\\
    &\quad+\frac{1}{2}\bigl|{\bf{\Sigma}}(\theta_0)^{-1}\bigr|
    \Bigl|\sqrt{n}\bigl(\check{\bf{\Sigma}}_n-{\bf{\Sigma}}(\theta_0)\bigr)\Bigr|\\
    &\quad+\frac{1}{2}\bigl|\log\det {\bf{\Sigma}}(\theta_0)\bigr|\Bigl|\sqrt{n}\Bigl(\frac{N_n}{n}-1\Bigr)\Bigr|=O_p(1),
\end{align*}
which implies (\ref{Yprob}). This completes the proof.
\end{proof}
\begin{lemma}\label{pilemma}
Under $\bf{[A1]}$-$\bf{[A4]}$, $\bf{[B]}$ and $\bf{[C1]}$,
\begin{align*}
    \sup_{u\in\hat{\mathbb{U}}_n\cap A_n}\bigl|\pi(\hat{\theta}_n+n^{-1/2}u)-\pi(\hat{\theta}_n)\bigr|
    \stackrel{p}{\longrightarrow}0
\end{align*}
as $n\longrightarrow\infty$.
\end{lemma}
\begin{proof}
It follows from {\bf{[B]}}, Lemma \ref{thetahat} and the continuous mapping theorem that 
\begin{align*}
    \pi(\hat{\theta}_n)\stackrel{p}{\longrightarrow}\pi(\theta_0).
\end{align*}
Since
\begin{align*}
    &\quad\ \sup_{u\in\hat{\mathbb{U}}_n\cap A_n}\bigl|\pi(\hat{\theta}_n+n^{-1/2}u)-\pi(\hat{\theta}_n)\bigr|\\
    &\leq\sup_{u\in\hat{\mathbb{U}}_n\cap A_n}\bigl|\pi(\hat{\theta}_n+n^{-1/2}u)-\pi(\theta_0)\bigr|+\bigl|\pi(\theta_0)-\pi(\hat{\theta}_n)\bigr|,
\end{align*}
it is sufficient to show
\begin{align}
    \sup_{u\in\hat{\mathbb{U}}_n\cap A_n}\bigl|\pi(\hat{\theta}_n+n^{-1/2}u)-\pi(\theta_0)\bigr|
    \stackrel{p}{\longrightarrow}0. \label{pihat}
\end{align}
Fix a number $\varepsilon>0$ arbitrarily. Since $\pi(\theta)$ is continuous, 
there exists a constant $\delta>0$ such that
\begin{align*}
    |\theta-\theta_0|\leq\delta\Longrightarrow|\pi(\theta)-\pi(\theta_0)|\leq\varepsilon,
\end{align*}
which implies
\begin{align}
    \sup_{|\theta-\theta_0|\leq\delta}|\pi(\theta)-\pi(\theta_0)|\leq\varepsilon. \label{supcon}
\end{align}
It holds from Lemma \ref{thetahat} that 
\begin{align*}
    \hat{u}_n=\sqrt{n}(\hat{\theta}_n-\theta_0)=O_p(1), 
\end{align*}
and hence there exist $M^{(1)}_0>0$ and $n_0^{(1)}\in\mathbb{N}$ such that 
\begin{align}
    {\bf{P}}\Bigl(|\hat{u}_n|>M^{(1)}_0\Bigr)<\varepsilon \label{PM0}
\end{align}
for all $n\geq n_0^{(1)}$. On 
\begin{align*}
    \Bigl\{|\hat{u}_n|\leq M^{(1)}_0\Bigr\},
\end{align*}
we see
\begin{align*}
    \bigl|\hat{\theta}_n+n^{-1/2}u-\theta_0\bigr|&=n^{-1/2}\bigl|\hat{u}_n+u\bigr|\\
    &\leq n^{-1/2}|\hat{u}_n|+n^{-1/2}|u|\leq n^{-1/2}M^{(1)}_0+n^{-1/8}
\end{align*}
for all $u\in A_n$. Furthermore, there exists $n_0^{(2)}\in\mathbb{N}$ such that 
\begin{align*}
    n^{-1/2}M^{(1)}_0+n^{-1/8}\leq\delta
\end{align*}
for all $n\geq n_0^{(2)}$ since $n^{-1/2}M^{(1)}_0+n^{-1/8}\longrightarrow 0$ as $n\longrightarrow\infty$. Hence, on
\begin{align*}
    \Bigl\{|\hat{u}_n|\leq M^{(1)}_0\Bigr\}
\end{align*}
it follows that
\begin{align*}
    \bigl|\hat{\theta}_n+n^{-1/2}u-\theta_0\bigr|\leq\delta
\end{align*}
for all $n\geq n_0^{(2)}$ and all $u\in A_n$, so that
\begin{align}
    \sup_{u\in\hat{\mathbb{U}}_n\cap A_n}\bigl|\pi(\hat{\theta}_n+n^{-1/2}u)-\pi(\theta_0)\bigr|\leq\sup_{|\theta-\theta_0|\leq\delta}\bigl|\pi(\theta)-\pi(\theta_0)\bigr|\label{supine}
\end{align}
for all $n\geq n_0^{(2)}$. Therefore, (\ref{supcon}), (\ref{PM0}) and (\ref{supine}) show
\begin{align*}
    &\quad\ {\bf{P}}\biggl(\sup_{u\in\hat{\mathbb{U}}_n\cap A_n}\bigl|\pi(\hat{\theta}_n+n^{-1/2}u)-\pi(\theta_0)\bigr|>\varepsilon\biggr)\\
    &={\bf{P}}\biggl(\biggl\{\sup_{u\in\hat{\mathbb{U}}_n\cap A_n}\bigl|\pi(\hat{\theta}_n+n^{-1/2}u)-\pi(\theta_0)\bigr|>\varepsilon\biggr\}\cap \Bigl\{|\hat{u}_n|\leq M^{(1)}_0\Bigr\}\biggr)\\
    &\qquad\qquad+{\bf{P}}\biggl(\biggl\{\sup_{u\in\hat{\mathbb{U}}_n\cap A_n}\bigl|\pi(\hat{\theta}_n+n^{-1/2}u)-\pi(\theta_0)\bigr|>\varepsilon\biggr\}\cap \Bigl\{|\hat{u}_n|> M^{(1)}_0\Bigr\}\biggr)\\
    &\leq{\bf{P}}\biggl(\biggl\{\sup_{|\theta-\theta_0|\leq\delta}\bigl|\pi(\theta)-\pi(\theta_0)\bigr|>\varepsilon\biggr\}\cap\Bigl\{|\hat{u}_n|\leq M^{(1)}_0\Bigr\}\biggr)+{\bf{P}}\Bigl(|\hat{u}_n|> M^{(1)}_0\Bigr)\\
    &\leq{\bf{P}}\biggl(\sup_{|\theta-\theta_0|\leq\delta}\bigl|\pi(\theta)-\pi(\theta_0)\bigr|>\varepsilon\biggr)+{\bf{P}}\Bigl(|\hat{u}_n|> M^{(1)}_0\Bigr)\\
    &=0+{\bf{P}}\Bigl(|\hat{u}_n|> M^{(1)}_0\Bigr)<\varepsilon
\end{align*}
for all $n\geq\max\{n_0^{(1)},n_0^{(2)}\}$, which yields (\ref{pihat}). This completes the proof.
\end{proof}
\begin{proposition}\label{Anprop1}
Under $\bf{[A1]}$-$\bf{[A4]}$, $\bf{[B]}$ and $\bf{[C1]}$,
\begin{align*}
    \int_{\hat{\mathbb{U}}_n\cap A_n}\hat{{\bf{Z}}}_n(u)\bigl\{\pi(\hat{\theta}_n+n^{-1/2}u)-\pi(\hat{\theta}_n)\bigr\}du\stackrel{p}{\longrightarrow}0
\end{align*}
as $n\longrightarrow\infty$.
\end{proposition}
\begin{proof}
Since the true parameter $\theta_0$ belongs to the open set $\Theta$, Lemma \ref{thetahat} implies
\begin{align*}
    {\bf{P}}\Bigl(\hat{\theta}_n\in\Theta\Bigr)\longrightarrow 1
\end{align*}
as $n\longrightarrow\infty$. From Lemma \ref{minlambda}, there exists $\lambda_0>0$ such that 
\begin{align*}
    {\bf{P}}\Bigl(\lambda_{min}\, {\bf{\Gamma}}(\hat{\theta}_n)\geq 4\lambda_0\Bigr)\longrightarrow 1
\end{align*}
as $n\longrightarrow\infty$. Moreover, it follows from Lemma \ref{problemma} that
\begin{align*}
    \bigl|{\bf{\Gamma}}_n(\hat{\theta}_n)-{\bf{\Gamma}}(\hat{\theta}_n)\bigr|\leq\sup_{\theta\in\Theta}\bigl|{\bf{\Gamma}}_n(\theta)-{\bf{\Gamma}}(\theta)\bigr|\stackrel{p}{\longrightarrow}0
\end{align*}
and
\begin{align*}
    n^{-1/8}\times\frac{1}{n}\sup_{\theta\in\Theta}\bigl|\partial_{\theta}^3{\bf{H}}_n(\theta)\bigr|\stackrel{p}{\longrightarrow}0,
\end{align*}
which implies
\begin{align}
    {\bf{P}}\bigl(G_{1,n}\bigr)\longrightarrow 1 \label{Gprob}
\end{align}
as $n\longrightarrow\infty$, where 
\begin{align*}
    G_{1,n}&=\Bigl\{\hat{\theta}_n\in \Theta\Bigr\}\cap\Bigl\{\lambda_{min}\, {\bf{\Gamma}}(\hat{\theta}_n)\geq 4\lambda_0\Bigr\}\\
    &\qquad\cap\Bigl\{\bigl|{\bf{\Gamma}}_n(\hat{\theta}_n)-{\bf{\Gamma}}(\hat{\theta}_n)\bigr|\leq \lambda_0\Bigr\}\cap\biggl\{n^{-1/8}\times\frac{1}{n}\sup_{\theta\in\Theta}\bigl|\partial_{\theta}^3{\bf{H}}_n(\theta)\bigr|\leq 3\lambda_0\biggr\}.
\end{align*}
Using Taylor’s theorem, on $G_{1,n}$, we have
\begin{align}
\begin{split}
    \log\hat{\bf{Z}}_n(u)
    &={\bf{H}}_n(\hat{\theta}_n+n^{-1/2}u)-{\bf{H}}_n(\hat{\theta}_n)\\
    &=n^{-1/2}\partial_{\theta}{\bf{H}}_n(\hat{\theta}_n)^{\top}u-\frac{1}{2}u^{\top}{\bf{\Gamma}}_n(\hat{\theta}_n)u+{\bf{R}}_{n}(u)\\
    &=-\frac{1}{2}u^{\top}{\bf{\Gamma}}(\hat{\theta}_n)u-\frac{1}{2}u^{\top}\bigl\{{\bf{\Gamma}}_n(\hat{\theta}_n)-{\bf{\Gamma}}(\hat{\theta}_n)\bigr\}u+{\bf{R}}_{n}(u)\label{ZhatTaylor}
\end{split}
\end{align}
for all $u\in\hat{\mathbb{U}}_n\cap A_n$, where 
\begin{align*}
    \check{\theta}_{n,\lambda,u}=\hat{\theta}_n+\lambda  n^{-1/2}u
\end{align*}
for $\lambda\in [0,1]$ and 
\begin{align*}
    {\bf{R}}_{n}(u)=\sum_{i=1}^q\sum_{j=1}^q\sum_{k=1}^q\frac{1}{2n^{3/2}}\int_0^1(1-\lambda)^2\partial_{\theta^{(i)}}\partial_{\theta^{(j)}}\partial_{\theta^{(k)}}{\bf{H}}_n(\check{\theta}_{n,\lambda,u})d\lambda u^{(i)}u^{(j)}u^{(k)}.
\end{align*}
On $G_{1,n}$, it follows that
\begin{align*}
    \hat{\theta}_n\in\Theta,\quad \hat{\theta}_n+n^{-1/2}u\in\Theta
\end{align*}
for all $u\in\hat{\mathbb{U}}_n\cap A_n$, and $\Theta$ is convex, so that (\ref{ZhatTaylor}) is well-defined. Since 
\begin{align*}
    4\lambda_0|u|^2\leq\lambda_{min}\, {\bf{\Gamma}}(\hat{\theta}_n)|u|^2\leq
    u^{\top}{\bf{\Gamma}}(\hat{\theta}_n)u
\end{align*}
for all $u\in\mathbb{R}^q$ on $G_{1,n}$, we obtain
\begin{align}
    -\frac{1}{2}u^{\top}{\bf{\Gamma}}(\hat{\theta}_n)u\leq -2\lambda_0|u|^2 \label{UAP-2-1}
\end{align}
and
\begin{align}
\begin{split}
    -\frac{1}{2}u^{\top}\bigl\{{\bf{\Gamma}}_n(\hat{\theta}_n)-{\bf{\Gamma}}(\hat{\theta}_n)\bigr\}u&\leq \frac{1}{2}\bigl|u^{\top}\bigl\{{\bf{\Gamma}}_n(\hat{\theta}_n)-{\bf{\Gamma}}(\hat{\theta}_n)\bigr\}u\bigr|\\
    &\leq\frac{1}{2}\bigl|{\bf{\Gamma}}_n(\hat{\theta}_n)-{\bf{\Gamma}}(\hat{\theta}_n)\bigr||u|^2\leq \frac{\lambda_0}{2}|u|^2
\end{split}\label{UAP-2-2}
\end{align}
for all $u\in\mathbb{R}^q$ on $G_{1,n}$. In addition, on $G_{1,n}$ one has
\begin{align*}
    \bigl|{\bf{R}}_{n}(u)\bigr|
    &\leq\sum_{i=1}^q\sum_{j=1}^q\sum_{k=1}^q\frac{1}{2n^{3/2}}\int_0^1(1-\lambda)^2\bigl|\partial_{\theta^{(i)}}\partial_{\theta^{(j)}}\partial_{\theta^{(k)}}{\bf{H}}_n(\check{\theta}_{n,\lambda,u})\bigr|d\lambda|u^{(i)}||u^{(j)}||u^{(k)}|\\
    &\leq\frac{1}{6n^{3/2}}\sum_{i=1}^q\sum_{j=1}^q\sum_{k=1}^q\sup_{\theta\in\Theta}\bigl|\partial_{\theta^{(i)}}\partial_{\theta^{(j)}}\partial_{\theta^{(k)}}{\bf{H}}_n(\theta)\bigr||u^{(i)}||u^{(j)}||u^{(k)}|\\
    &\leq\frac{1}{6n^{3/2}}
    \Biggl(\sum_{i=1}^q\sum_{j=1}^q\sum_{k=1}^q
    \sup_{\theta\in\Theta}\bigl|\partial_{\theta^{(i)}}\partial_{\theta^{(j)}}
    \partial_{\theta^{(k)}}{\bf{H}}_n(\theta)\bigr|^2\Biggr)^{1/2}\\
    &\qquad\qquad\qquad\qquad\qquad
    \times\Biggl(\sum_{i=1}^q\sum_{j=1}^q\sum_{k=1}^q|u^{(i)}|^2|u^{(j)}|^2|u^{(k)}|^2\Biggr)^{1/2}\\
    &=\frac{1}{6n^{3/2}}\sup_{\theta\in\Theta}
    \bigl|\partial^3_{\theta}{\bf{H}}_n(\theta)\bigr|
    \Biggl(\sum_{i=1}^q|u^{(i)}|^2\Biggr)^{1/2}
    \Biggl(\sum_{j=1}^q|u^{(j)}|^2\Biggr)^{1/2}
    \Biggl(\sum_{k=1}^q|u^{(k)}|^2\Biggr)^{1/2}\\
    &=\frac{1}{6}\times\biggl(\frac{1}{n}\sup_{\theta\in\Theta}
    \bigl|\partial^3_{\theta}{\bf{H}}_n(\theta)\bigr|\biggr)\times n^{-1/2}|u|\times|u|^2\\
    &\leq\frac{1}{6}\times3n^{1/8}\lambda_0\times n^{-1/8}\times |u|^2=\frac{\lambda_0}{2}|u|^2
\end{align*}
for all $u\in\hat{\mathbb{U}}_n\cap A_n$, so that 
\begin{align}
    \bigl|{\bf{R}}_{n}(u)\bigr|\leq \frac{\lambda_0}{2}|u|^2 \label{UAP-2-3}
\end{align}
for all $u\in\hat{\mathbb{U}}_n\cap A_n$. Hence, on $G_{1,n}$, it holds from (\ref{ZhatTaylor})-(\ref{UAP-2-3}) that 
\begin{align*}
    \log\hat{\bf{Z}}_n(u)\leq-2\lambda_0|u|^2+\frac{\lambda_0}{2}|u|^2+\frac{\lambda_0}{2}|u|^2=-\lambda_0|u|^2,
\end{align*}
for all $u\in\hat{\mathbb{U}}_n\cap A_n$, which yields
\begin{align}
    \hat{\bf{Z}}_n(u)\leq\exp\Bigl\{-\lambda_0|u|^2\Bigr\} \label{Zhatine}
\end{align}
for all $u\in\hat{\mathbb{U}}_n\cap A_n$. Since
\begin{align*}
    \int_{\mathbb{R}^q}\exp\Bigl\{-\lambda_0|u|^2\Bigr\}du=(2\pi)^{q/2}(2\lambda_0)^{-q/2}<\infty,
\end{align*}
we have
\begin{align*}
    &\quad\ \biggr|\int_{\hat{\mathbb{U}}_n\cap A_n}\hat{\bf{Z}}_n(u)\bigl\{\pi(\hat{\theta}_n+n^{-1/2}u)-\pi(\hat{\theta}_n)\bigr\}du\biggl|\\
    &\leq \int_{\hat{\mathbb{U}}_n\cap A_n}\hat{\bf{Z}}_n(u)\bigl|\pi(\hat{\theta}_n+n^{-1/2}u)-\pi(\hat{\theta}_n)\bigr|du\\
    &\leq\sup_{u\in\hat{\mathbb{U}}_n\cap A_n}\bigl|\pi(\hat{\theta}_n+n^{-1/2}u)-\pi(\hat{\theta}_n)
    \bigr|\int_{\hat{\mathbb{U}}_n\cap A_n}
    \exp\Bigl\{-\lambda_0|u|^2\Bigr\}du\\
    &\leq\sup_{u\in\hat{\mathbb{U}}_n\cap A_n}\bigl|\pi(\hat{\theta}_n+n^{-1/2}u)-\pi(\hat{\theta}_n)
    \bigr|\int_{\mathbb{R}^q}
    \exp\Bigl\{-\lambda_0|u|^2\Bigr\}du\\
    &\leq C\sup_{u\in\hat{\mathbb{U}}_n\cap A_n}\bigl|\pi(\hat{\theta}_n+n^{-1/2}u)-\pi(\hat{\theta}_n)
    \bigr|
\end{align*}
on $G_{1,n}$. Therefore, as it holds from Lemma \ref{pilemma} and (\ref{Gprob}) that
\begin{align*}
    &\quad\ {\bf{P}}\Biggl(\biggr|\int_{\hat{\mathbb{U}}_n\cap A_n}\hat{\bf{Z}}_n(u)\bigl\{\pi(\hat{\theta}_n+n^{-1/2}u)-\pi(\hat{\theta}_n)\bigr\}du\biggl|>\varepsilon\Biggr)\\
    &\leq{\bf{P}}\Biggl(\biggl\{\biggr|\int_{\hat{\mathbb{U}}_n\cap A_n}\hat{\bf{Z}}_n(u)\bigl\{\pi(\hat{\theta}_n+n^{-1/2}u)-\pi(\hat{\theta}_n)\bigr\}du\biggl|>\varepsilon\biggr\}\cap G_{1,n}\Biggr)\\
    &\qquad+{\bf{P}}\Biggl(\biggl\{\biggr|\int_{\hat{\mathbb{U}}_n\cap A_n}\hat{\bf{Z}}_n(u)\bigl\{\pi(\hat{\theta}_n+n^{-1/2}u)-\pi(\hat{\theta}_n)\bigr\}du\biggl|>\varepsilon\biggr\}\cap G^c_{1,n}\Biggr)\\
    &\leq {\bf{P}}\Biggl(\biggl\{C\sup_{u\in\hat{\mathbb{U}}_n\cap A_n}\bigl|\pi(\hat{\theta}_n+n^{-1/2}u)-\pi(\hat{\theta}_n)
    \bigr|>\varepsilon\biggr\}\cap G_{1,n}\Biggr)+{\bf{P}}\bigl(G_{1,n}^c\bigr)\\
    &\leq {\bf{P}}\biggl(C\sup_{u\in\hat{\mathbb{U}}_n\cap A_n}\bigl|\pi(\hat{\theta}_n+n^{-1/2}u)-\pi(\hat{\theta}_n)\bigr|>\varepsilon\biggr)+{\bf{P}}\bigl(G_{1,n}^c\bigr)\longrightarrow 0
\end{align*}
as $n\longrightarrow\infty$ for all $\varepsilon>0$, we obtain
\begin{align*}
    \int_{\hat{\mathbb{U}}_n\cap A_n}\hat{\bf{Z}}_n(u)\bigl\{\pi(\hat{\theta}_n+n^{-1/2}u)-\pi(\hat{\theta}_n)\bigr\}du\stackrel{p}{\longrightarrow}0,
\end{align*}
which completes the proof.
\end{proof}
\begin{proposition}\label{Anprop2}
Under $\bf{[A1]}$-$\bf{[A4]}$ and $\bf{[C1]}$,
\begin{align*}
    \int_{\hat{\mathbb{U}}_n\cap A_n}\hat{\bf{Z}}_n(u)du\stackrel{p}{\longrightarrow}\int_{\mathbb{R}^q}{\bf{Z}}(u)du
\end{align*}
as $n\longrightarrow\infty$.
\end{proposition}
\begin{proof}
If it follows that
\begin{align}
    \int_{\hat{\mathbb{U}}_n\cap A_n}\hat{\bf{Z}}_n(u)du-\int_{\hat{\mathbb{U}}_n\cap A_n}{\bf{Z}}(u)du\stackrel{p}{\longrightarrow}0 \label{intAnprob-1}
\end{align}
and
\begin{align}
    \int_{\hat{\mathbb{U}}_n\cap A_n}{\bf{Z}}(u)du\stackrel{p}{\longrightarrow}\int_{\mathbb{R}^q}{\bf{Z}}(u)du,\label{intAnprob-2}
\end{align}
Slutsky’s theorem shows 
\begin{align*}
    \int_{\hat{\mathbb{U}}_n\cap A_n}\hat{\bf{Z}}_n(u)du\stackrel{p}{\longrightarrow}\int_{\mathbb{R}^q}{\bf{Z}}(u)du.
\end{align*}
Therefore, it is sufficient to prove (\ref{intAnprob-1}) and (\ref{intAnprob-2}),
for which we follow the argument used in the proof of Theorem 1 in Jasra et al. \cite{Jasra(2019)}. \\
\ \\
\noindent
Proof of (\ref{intAnprob-1}). Recall that $\tilde{\bf{Z}}_n(u)=\hat{\bf{Z}}_n(u)$ for all $u\in \hat{\mathbb{U}}_n\cap A_n$. Since
\begin{align*}
    \biggl|\int_{\hat{\mathbb{U}}_n\cap A_n}\hat{\bf{Z}}_n(u)du-\int_{\hat{\mathbb{U}}_n\cap A_n}{\bf{Z}}(u)du\biggr|
    &\leq\int_{\hat{\mathbb{U}}_n\cap A_n}\bigl|\hat{\bf{Z}}_n(u)-{\bf{Z}}(u)\bigr|du\\
    &=\int_{\hat{\mathbb{U}}_n\cap A_n}\bigl|\tilde{\bf{Z}}_n(u)-{\bf{Z}}(u)\bigr|du\\
    &=\int_{\mathbb{R}^q}\bigl|\tilde{\bf{Z}}_n(u)-{\bf{Z}}(u)\bigr|{1}_{\hat{\mathbb{U}}_n\cap A_n}(u)du,
\end{align*}
it is enough to show
\begin{align*}
    \int_{\mathbb{R}^q}\bigl|\tilde{\bf{Z}}_n(u)-{\bf{Z}}(u)\bigr|{1}_{\hat{\mathbb{U}}_n\cap A_n}(u)du\stackrel{p}{\longrightarrow}0
\end{align*}
to prove (\ref{intAnprob-1}). Note that
\begin{align*}
    \int_{\mathbb{R}^q}\bigl|\tilde{\bf{Z}}_n(u)-{\bf{Z}}(u)\bigr|{1}_{\hat{\mathbb{U}}_n\cap A_n}(u)du&=\int_{\mathbb{R}^q}\bigl|\tilde{\bf{Z}}_n(u)-{\bf{Z}}(u)\bigr|{1}_{\hat{\mathbb{U}}_n\cap A_n}(u){1}_{G_{1,n}}du\\
    &\qquad+\int_{\mathbb{R}^q}\bigl|\tilde{\bf{Z}}_n(u)-{\bf{Z}}(u)\bigr|{1}_{\hat{\mathbb{U}}_n\cap A_n}(u){1}_{G^c_{1,n}}du.
\end{align*}
By (\ref{Gprob}), we have
\begin{align*}
    &\quad\ {\bf{P}}\biggl(\biggl|\int_{\mathbb{R}^q}\bigl|\tilde{\bf{Z}}_n(u)-{\bf{Z}}(u)\bigr|{1}_{\hat{\mathbb{U}}_n\cap A_n}(u){1}_{G_{1,n}^c}du\biggr|>\varepsilon\biggr)\\
    &\leq {\bf{P}}\biggl(\biggl\{\biggl|\int_{\mathbb{R}^q}\bigl|\tilde{\bf{Z}}_n(u)-{\bf{Z}}(u)\bigr|{1}_{\hat{\mathbb{U}}_n\cap A_n}(u){1}_{G_{1,n}^c}du\biggr|>\varepsilon\biggr\}\cap G_{1,n}\biggr)\\
    &\qquad+{\bf{P}}\biggl(\biggl\{\biggl|\int_{\mathbb{R}^q}\bigl|\tilde{\bf{Z}}_n(u)-{\bf{Z}}(u)\bigr|{1}_{\hat{\mathbb{U}}_n\cap A_n}(u){1}_{G_{1,n}^c}du\biggr|>\varepsilon\biggr\}\cap G^c_{1,n}\biggr)\\
    &\leq{\bf{P}}\Bigl(\bigl\{0>\varepsilon\bigr\}\cap G_{1,n}\Bigr)+{\bf{P}}\bigl(G_{1,n}^c\bigr)\\
    &\leq{\bf{P}}\bigl(0>\varepsilon\bigr)+{\bf{P}}\bigl(G_{1,n}^c\bigr)={\bf{P}}\bigl(G_{1,n}^c\bigr)\longrightarrow 0
\end{align*}
as $n\longrightarrow\infty$ for all $\varepsilon>0$, which yields
\begin{align*}
    \int_{\mathbb{R}^q}\bigl|\tilde{\bf{Z}}_n(u)-{\bf{Z}}(u)\bigr|{1}_{\hat{\mathbb{U}}_n\cap A_n}(u){1}_{G_{1,n}^c}du\stackrel{p}{\longrightarrow}0.
\end{align*}
Consequently, it remains to prove
\begin{align}
    S_{1,n}=\int_{\mathbb{R}^q}\bigl|\tilde{\bf{Z}}_n(u)-{\bf{Z}}(u)\bigr|{1}_{\hat{\mathbb{U}}_n\cap A_n}(u){1}_{G_{1,n}}du\stackrel{p}{\longrightarrow}0.\label{S1}
\end{align}
To show (\ref{S1}), we use a subsequence argument. Let $\mathbb{N}'\subset\mathbb{N}$ be an arbitrary infinite subsequence. It is sufficient to show that there exists a further subsequence $\mathbb{N}''=\{n''\}\subset\mathbb{N}'$ such that
\begin{align}
    S_{1,n''}\stackrel{a.s.}{\longrightarrow}0. \label{barS}
\end{align}
First of all, we set
\begin{align*}
    M_{1,n}(u)=\bigl|\tilde{\bf{Z}}_n(u)-{\bf{Z}}(u)\bigr|{1}_{\hat{\mathbb{U}}_n\cap A_n}(u){1}_{G_{1,n}}
\end{align*}
for $u\in\mathbb{R}^q$. On $G_{1,n}$, it is shown that
\begin{align*}
    \log\tilde{\bf{Z}}_n(u)-\log{\bf{Z}}(u)&=
    \log\hat{\bf{Z}}_n(u)-\log{\bf{Z}}(u)\\
    &=-\frac{1}{2}u^{\top}\bigl\{{\bf{\Gamma}}_n(\hat{\theta}_n)-{\bf{\Gamma}}(\theta_0)\bigr\}u+{\bf{R}}_{n}(u)
\end{align*}
and
\begin{align*}
    |{\bf{R}}_{n}(u)|\leq\frac{n^{-1/8}}{6}\biggl\{\frac{1}{n}\sup_{\theta\in\Theta}
    \bigl|\partial^3_{\theta}{\bf{H}}_n(\theta)\bigr|\biggr\}|u|^2
\end{align*}
for all $u\in\hat{\mathbb{U}}_n\cap A_n$ in an analogous way to the proof of Proposition \ref{Anprop1}. Since 
\begin{align*}
    \Bigl|-\frac{1}{2}u^{\top}\bigl\{{\bf{\Gamma}}_n(\hat{\theta}_n)-{\bf{\Gamma}}(\theta_0)\bigr\}u\Bigr|&\leq\frac{|u|^2}{2}\bigl|{\bf{\Gamma}}_n(\hat{\theta}_n)-{\bf{\Gamma}}(\theta_0)\bigr|\\
    &\leq\frac{|u|^2}{2}\sup_{\theta\in\Theta}\bigl|{\bf{\Gamma}}_n(\theta)-{\bf{\Gamma}}(\theta)\bigr|+\frac{|u|^2}{2}\bigl|{\bf{\Gamma}}(\hat{\theta}_n)-{\bf{\Gamma}}(\theta_0)\bigr|,
\end{align*}
on $G_{1,n}$, we have
\begin{align}
    \bigl|\log\tilde{\bf{Z}}_n(u)-\log{\bf{Z}}(u)\bigr|\leq C|u|^2Q_{n} \label{logZine}
\end{align}
for all $u\in \hat{\mathbb{U}}_n\cap A_n$, where
\begin{align*}
    Q_{n}=\sup_{\theta\in\Theta}\bigl|{\bf{\Gamma}}_n(\theta)-{\bf{\Gamma}}(\theta)\bigr|+\bigl|{\bf{\Gamma}}(\hat{\theta}_n)-{\bf{\Gamma}}(\theta_0)\bigr|+n^{-1/8}\biggl\{\frac{1}{n}\sup_{\theta\in\Theta}
    \bigl|\partial^3_{\theta}{\bf{H}}_n(\theta)\bigr|\biggr\}.
\end{align*}
Since Lemmas \ref{thetahat} and \ref{problemma} imply
\begin{align*}
    Q_{n}\stackrel{p}{\longrightarrow}0,
\end{align*}
there exists a further subsequence $\mathbb{N}''=\{n''\}\subset\mathbb{N}'$ such that
\begin{align}
    Q_{n''}\stackrel{a.s.}{\longrightarrow}0. \label{Ras}
\end{align}
As it holds from (\ref{Zhatine}) that
\begin{align*}
     \tilde{\bf{Z}}_n(u){1}_{\hat{\mathbb{U}}_n\cap A_n}(u){1}_{G_{1,n}}\leq \exp\Bigl\{-\lambda_0|u|^2\Bigr\}{1}_{\hat{\mathbb{U}}_n\cap A_n}(u){1}_{G_{1,n}}
\end{align*}
for all $u\in\mathbb{R}^q$, one gets
\begin{align*}
    M_{1,n}(u)&\leq
    \tilde{\bf{Z}}_{n}(u){1}_{\hat{\mathbb{U}}_n\cap A_n}(u){1}_{G_{1,n}}+{\bf{Z}}(u){1}_{\hat{\mathbb{U}}_n\cap A_n}(u){1}_{G_{1,n}}\\
    &\leq\exp\Bigl\{-\lambda_0|u|^2\Bigr\}{1}_{\hat{\mathbb{U}}_n\cap A_n}(u){1}_{G_{1,n}}+{\bf{Z}}(u)\\
    &\leq\exp\Bigl\{-\lambda_0|u|^2\Bigr\}+{\bf{Z}}(u)
\end{align*}
and
\begin{align*}
    &\quad\ \int_{\mathbb{R}^q}\exp\Bigl\{-\lambda_0|u|^2\Bigr\}du+\int_{\mathbb{R}^q}{\bf{Z}}(u)du\\
    &=(2\pi)^{q/2}(2\lambda_0)^{-q/2}+(2\pi)^{q/2}|\det {\bf{\Gamma}}(\theta_0)|^{-1/2}<\infty
\end{align*}
for all $u\in\mathbb{R}^q$. Since 
\begin{align*}
    |e^{x}-1|\leq |x|e^{|x|}
\end{align*}
for all $x\in\mathbb{R}$, we see from (\ref{logZine}) that 
\begin{align*}
    M_{1,n}(u)&=\bigl|\tilde{\bf{Z}}_n(u)-{\bf{Z}}(u)\bigr|{1}_{\hat{\mathbb{U}}_n\cap A_n}(u){1}_{G_{1,n}}\\
    &\leq{\bf{Z}}(u)\biggl|\frac{\tilde{\bf{Z}}_n(u)}{{\bf{Z}}(u)}-1\biggr|{1}_{\hat{\mathbb{U}}_n\cap A_n}(u){1}_{G_{1,n}}\\
    &={\bf{Z}}(u)\biggl|\exp\Bigl\{\log \tilde{\bf{Z}}_n(u)-\log{\bf{Z}}(u)\Bigr\}-1\biggr|{1}_{\hat{\mathbb{U}}_n\cap A_n}(u){1}_{G_{1,n}}\\
    &\leq{\bf{Z}}(u)\bigl|\log \tilde{\bf{Z}}_n(u)-\log{\bf{Z}}(u)\bigr|\exp\Bigl\{\bigl|\log \tilde{\bf{Z}}_n(u)-\log{\bf{Z}}(u)\bigr|\Bigr\}{1}_{\hat{\mathbb{U}}_n\cap A_n}(u){1}_{G_{1,n}}\\
    &\leq C|u|^2Q_{n}\exp\Bigl\{C|u|^2Q_{n}\Bigr\}{\bf{Z}}(u){1}_{\hat{\mathbb{U}}_n\cap A_n}(u){1}_{G_{1,n}}\\
    &\leq C|u|^2Q_{n}\exp\Bigl\{C|u|^2Q_{n}\Bigr\}{\bf{Z}}(u)
\end{align*}
for all $u\in\mathbb{R}^q$. Therefore, on
\begin{align*}
    \Bigl\{\lim_{n''\to\infty}Q_{n''}=0\Bigr\},
\end{align*}
for all $u\in\mathbb{R}^q$, we  have
\begin{align*}
     M_{1,n''}(u)&\leq\exp\Bigl\{-\lambda_0|u|^2\Bigr\}+{\bf{Z}}(u),\ \int_{\mathbb{R}^q}\exp\Bigl\{-\lambda_0|u|^2\Bigr\}+{\bf{Z}}(u)du<\infty
\end{align*}
and 
\begin{align*}
     (0\leq)M_{1,n''}(u)\leq C|u|^2Q_{n''}\exp\Bigl\{C|u|^2Q_{n''}\Bigr\}{\bf{Z}}(u)\longrightarrow 0
\end{align*}
as $n''\longrightarrow\infty$ , so that
\begin{align*}
    \lim_{n''\to\infty}S_{1,n''}&=\lim_{n''\to\infty}
    \int_{\mathbb{R}^q}M_{1,n''}(u)du\\
    &=\int_{\mathbb{R}^q}\lim_{n''\to\infty}M_{1,n''}(u)du=0
\end{align*}
by the dominated convergence theorem, which yields
\begin{align*}
    \Bigl\{\lim_{n''\to\infty}Q_{n''}=0\Bigr\}\subset\Bigl\{\lim_{n''\to\infty}S_{1,n''}=0\Bigr\}.
\end{align*}
Consequently, (\ref{Ras}) shows
\begin{align*}
    1={\bf{P}}\Bigl(\lim_{n''\to\infty}Q_{n''}=0\Bigr)\leq{\bf{P}}\Bigl(\lim_{n''\to\infty}S_{1,n''}=0\Bigr),
\end{align*}
so that (\ref{barS}) holds. This completes the proof.\\
\ \\
Proof of (\ref{intAnprob-2}). It is enough to prove that for all subsequence $\mathbb{N}'\subset\mathbb{N}$, there exists a further subsequence $\mathbb{N}''=\{n''\}\subset\mathbb{N}'$ such that
\begin{align}
    S_{2,n''}\stackrel{a.s.}{\longrightarrow}\int_{\mathbb{R}^q}{\bf{Z}}(u)du,  \label{bar2S}
\end{align}
where 
\begin{align*}
    S_{2,n}=\int_{\hat{\mathbb{U}}_n\cap A_n}{\bf{Z}}(u)du=\int_{\mathbb{R}^q}1_{\hat{\mathbb{U}}_n\cap A_n}(u){\bf{Z}}(u)du.
\end{align*}
Fix an arbitrary infinite subsequence $\mathbb{N}'\subset\mathbb{N}$. Let 
\begin{align*}
    M_{2,n}(u)=1_{\hat{\mathbb{U}}_n\cap A_n}(u){\bf{Z}}(u)
\end{align*}
for $u\in\mathbb{R}^q$. As it holds from Lemma \ref{thetahat} that 
\begin{align*}
    \hat{\theta}_n\stackrel{p}{\longrightarrow}\theta_0,
\end{align*}
there exists a further subsequence $\mathbb{N}''=\{n''\}\subset\mathbb{N}'$ such that
\begin{align}
    \hat{\theta}_{n''}\stackrel{a.s.}{\longrightarrow}\theta_0. \label{M2as}
\end{align}
First, we prove that 
\begin{align}
    \lim_{n''\to\infty}{1}_{\hat{\mathbb{U}}_{n''}\cap A_{n''}}(u)=1 \label{indi}
\end{align}
for all $u\in\mathbb{R}^q$ on
\begin{align*}
     \Bigl\{\lim_{n''\to\infty}\hat{\theta}_{n''}=\theta_0\Bigr\}.
\end{align*}
Fix $u_0\in\mathbb{R}^q$ arbitrarily. For all $n''\geq |u_0|^{8/3}$, we have
\begin{align*}
  {(n'')}^{-1/2}|u_0|\leq {(n'')}^{-1/8},
\end{align*}
which yields $u_0\in A_{n''}$. Note that there exists $\rho_0>0$ such that
\begin{align*}
    B(\theta_0,\rho_0)\subset\Theta
\end{align*}
since $\Theta$ is open, and there exists $n^{(3)}_0\in\mathbb{N}$ such that
\begin{align*}
    {(n'')}^{-1/2}|u_0|\leq\frac{\rho_0}{2}
\end{align*}
for all $n''\geq n^{(3)}_0$ since ${(n'')}^{-1/2} |u_0|\longrightarrow 0$ as $n''\longrightarrow\infty$. On 
\begin{align*}
     \Bigl\{\lim_{n''\to\infty}\hat{\theta}_{n''}=\theta_0\Bigr\},
\end{align*}
there exists $n^{(4)}_0\in\mathbb{N}$ such that
\begin{align*}
    |\hat{\theta}_{n''}-\theta_0|\leq \frac{\rho_0}{2}
\end{align*}
for all $n''\geq n^{(4)}_0$, so that
\begin{align*}
    \bigl|\hat{\theta}_{n''}+{(n'')}^{-1/2}u_0-\theta_0\bigr|\leq |\hat{\theta}_{n''}-\theta_0|+{(n'')}^{-1/2}|u_0|\leq\rho_0
\end{align*}
for all $n''\geq\max\{n_0^{(3)},n^{(4)}_0\}$, which deduces $u_0\in\hat{\mathbb{U}}_{n''}$ for all $n''\geq\max\{n_0^{(3)},n^{(4)}_0\}$. Hence, on 
\begin{align*}
     \Bigl\{\lim_{n''\to\infty}\hat{\theta}_{n''}=\theta_0\Bigr\},
\end{align*}
we obtain
\begin{align*}
    u_0\in\hat{\mathbb{U}}_{n''}\cap A_{n''}
\end{align*}
for all $n''\geq\max\{|u_0|^{8/3},n_0^{(3)},n^{(4)}_0\}$, which implies
\begin{align*}
    \lim_{n''\to\infty}{1}_{\hat{\mathbb{U}}_{n''}\cap A_{n''}}(u_0)=1.
\end{align*}
This completes the proof of (\ref{indi}). On
\begin{align*}
     \Bigl\{\lim_{n''\to\infty}\hat{\theta}_{n''}=\theta_0\Bigr\},
\end{align*}
it follows from (\ref{indi}) that
\begin{align*}
    M_{2,n''}(u)\leq{\bf{Z}}(u),\quad \int_{\mathbb{R}^q}{\bf{Z}}(u)du<\infty
\end{align*}
and
\begin{align*}
    \lim_{n''\to\infty}M_{2,n''}(u)={\bf{Z}}(u)\lim_{n''\to\infty}{1}_{\hat{\mathbb{U}}_{n''}\cap A_{n''}}(u)={\bf{Z}}(u)
\end{align*}
for all $u\in\mathbb{R}^q$, so that 
by the dominated convergence theorem, we have
\begin{align*}
    \lim_{n''\to\infty}S_{2,n''}&=\lim_{n''\to\infty}
    \int_{\mathbb{R}^q}M_{2,n''}(u)du\\
    &=\int_{\mathbb{R}^q}\lim_{n''\to\infty}M_{2,n''}(u)du=\int_{\mathbb{R}^q}{\bf{Z}}(u)du,
\end{align*}
which yields
\begin{align*}
    \Bigl\{\lim_{n''\to\infty}\hat{\theta}_{n''}=\theta_0\Bigr\}\subset\biggl\{\lim_{n''\to\infty}S_{2,n''}=\int_{\mathbb{R}^q}{\bf{Z}}(u)du\biggr\}.
\end{align*}
Consequently, since it follows from (\ref{M2as}) that 
\begin{align*}
    {\bf{P}}\biggl(\lim_{n''\to\infty}S_{2,n''}=\int_{\mathbb{R}^q}{\bf{Z}}(u)du\biggr)=1,
\end{align*}
we obtain (\ref{bar2S}). This completes the proof.
\end{proof}
\begin{proposition}\label{Ancprop}
Under $\bf{[A1]}$-$\bf{[A4]}$, $\bf{[B]}$ and $\bf{[C1]}$,
\begin{align*}
    \int_{\hat{\mathbb{U}}_n\cap A^c_n}
    \hat{\bf{Z}}_n(u)\pi(\hat{\theta}_n+n^{-1/2}u)du\stackrel{p}{\longrightarrow}0
\end{align*}
as $n\longrightarrow\infty$.
\end{proposition}
\begin{proof}
Since 
\begin{align*}
    0&\leq\int_{\hat{\mathbb{U}}_n\cap A^c_n}\hat{\bf{Z}}_n(u)\pi(\hat{\theta}_n+n^{-1/2}u)du\leq\sup_{\theta\in\Theta}\pi(\theta)\int_{\hat{\mathbb{U}}_n\cap A_n^c}\hat{\bf{Z}}_n(u)du,
\end{align*}
it is enough to show
\begin{align}
    \int_{\hat{\mathbb{U}}_n\cap A_n^c}\hat{\bf{Z}}_n(u)du\stackrel{p}{\longrightarrow}0.
    \label{intAncprob}
\end{align}
Fix an arbitrary number $\varepsilon>0$. As it holds from Lemmas \ref{thetahat} and \ref{problemma} that there exists $M^{(2)}_{0}>0$ and $n_0^{(5)}\in\mathbb{N}$ such that 
\begin{align*}
    {\bf{P}}\Bigl(|\hat{u}_n|>M^{(2)}_0\Bigr)<\frac{\varepsilon}{2}
\end{align*}
and
\begin{align*}
    {\bf{P}}\Bigl(\sqrt{n}\sup_{\theta\in\Theta}\bigl|{\bf{Y}}_n(\theta)-{\bf{Y}}(\theta)\bigr|>M^{(2)}_0\Bigr)<\frac{\varepsilon}{2}
\end{align*}
for all $n\geq n_0^{(5)}$, we have 
\begin{align}
    {\bf{P}}\bigl(G^c_{2,n}\bigr)<\varepsilon \label{UAcP-1}
\end{align}
for all $n\geq n_0^{(5)}$, where 
\begin{align*}
    G_{2,n}&=\Bigl\{|\hat{u}_n|\leq M^{(2)}_0\Bigr\}\cap\Bigl\{\sqrt{n}\sup_{\theta\in\Theta}\bigl|{\bf{Y}}_n(\theta)-{\bf{Y}}(\theta)\bigr|\leq M^{(2)}_0\Bigr\}.
\end{align*}
Since 
\begin{align*}
    n{\bf{Y}}_n(\hat{\theta}_n)={\bf{H}}_n(\hat{\theta}_n)-{\bf{H}}_n(\theta_0)\geq 0,
\end{align*}
by the definition of $\hat{\theta}_n$, it follows that
\begin{align*}
    \sup_{u\in\hat{\mathbb{U}}_n\cap A_n^c}\log\hat{\bf{Z}}_n(u)&=\sup_{u\in\hat{\mathbb{U}}_n\cap A_n^c}\Bigl\{{\bf{H}}_n(\hat{\theta}_n+n^{-1/2}u)-{\bf{H}}_n(\hat{\theta}_n)\Bigr\}\\
    &=\sup_{u\in\hat{\mathbb{U}}_n\cap A_n^c}\Bigl\{n{\bf{Y}}_n(\hat{\theta}_n+n^{-1/2}u)-n{\bf{Y}}_n(\hat{\theta}_n)\Bigr\}\\
    &=n\biggl\{\sup_{u\in\hat{\mathbb{U}}_n\cap A_n^c}{\bf{Y}}_n(\hat{\theta}_n+n^{-1/2}u)\biggr\}-n{\bf{Y}}_n(\hat{\theta}_n)\\
    &\leq n\sup_{u\in\hat{\mathbb{U}}_n\cap A_n^c}{\bf{Y}}_n(\hat{\theta}_n+n^{-1/2}u).
\end{align*}
Furthermore, we see
\begin{align*}
    &\quad\ n\sup_{u\in\hat{\mathbb{U}}_n\cap A_n^c}{\bf{Y}}_n(\hat{\theta}_n+n^{-1/2}u)\\
    &=n\sup_{u\in\hat{\mathbb{U}}_n\cap A_n^c}\Bigl\{{\bf{Y}}_n(\hat{\theta}_n+n^{-1/2}u)-{\bf{Y}}(\hat{\theta}_n+n^{-1/2}u)+{\bf{Y}}(\hat{\theta}_n+n^{-1/2}u)\Bigr\}\\
    &\leq n\sup_{u\in\hat{\mathbb{U}}_n\cap A_n^c}\Bigl\{{\bf{Y}}_n(\hat{\theta}_n+n^{-1/2}u)-{\bf{Y}}(\hat{\theta}_n+n^{-1/2}u)\Bigr\}+n\sup_{u\in\hat{\mathbb{U}}_n\cap A_n^c}{\bf{Y}}(\hat{\theta}_n+n^{-1/2}u)\\
    &\leq n\sup_{\theta\in\Theta}\bigl|{\bf{Y}}_n(\theta)-{\bf{Y}}(\theta)\bigr|
    +n\sup_{u\in\hat{\mathbb{U}}_n\cap A_n^c}{\bf{Y}}(\hat{\theta}_n+n^{-1/2}u)\\
    &\leq \sqrt{n}M^{(2)}_0+n\sup_{u\in\hat{\mathbb{U}}_n\cap A_n^c}{\bf{Y}}(\hat{\theta}_n+n^{-1/2}u)
\end{align*}
on $G_{2,n}$, so that 
\begin{align}
    \sup_{u\in\hat{\mathbb{U}}_n\cap A_n^c}\log\hat{\bf{Z}}_n(u)\leq 
    \sqrt{n}M^{(2)}_0+n\sup_{u\in\hat{\mathbb{U}}_n\cap A_n^c}{\bf{Y}}(\hat{\theta}_n+n^{-1/2}u)\label{UAcP-2-1}
\end{align}
on $G_{2,n}$. 
Next, we evaluate the second term on the right hand side of inequality \eqref{UAcP-2-1}.
On $G_{2,n}$, it holds that
\begin{align*}
    \bigl|\hat{\theta}_n+n^{-1/2}u-\theta_0\bigr|&=n^{-1/2}\bigl|\hat{u}_n+u\bigr|\\
    &\geq n^{-1/2}|u|-n^{-1/2}|\hat{u}_n|\\
    &\geq n^{-1/8}-n^{-1/2}M^{(2)}_0=n^{-1/8}\bigl(1-n^{-3/8}M^{(2)}_0\bigr)
\end{align*}
for all $u\in\hat{\mathbb{U}}_n\cap A_n^c$, and there exists $n_0^{(6)}\in\mathbb{N}$ such that
\begin{align*}
    n^{-3/8}M^{(2)}_0\leq 1/2
\end{align*}
for all $n\geq n_0^{(6)}$ since $n^{-3/8}M^{(2)}_0\longrightarrow 0$ as $n\longrightarrow\infty$. 
Thus, setting $\delta_n=\frac{n^{-1/8}}{2}$, we obtain that for all $n\geq n_0^{(6)}$, on $G_{2,n}$, 
\begin{align*}
    \bigl|\hat{\theta}_n+n^{-1/2}u-\theta_0\bigr|\geq\delta_n
\end{align*}
for all $u\in\hat{\mathbb{U}}_n\cap A_n^c$, so that 
\begin{align*}
    {\bf{Y}}(\hat{\theta}_n+n^{-1/2}u)\leq\sup_{|\theta-\theta_0|\geq\delta_n}{\bf{Y}}(\theta)
\end{align*}
for all $u\in\hat{\mathbb{U}}_n\cap A_n^c$.
Consequently, as it holds from [{\bf{C1}}] that
\begin{align*}
    \sup_{|\theta-\theta_0|\geq\delta_n}{\bf{Y}}(\theta)\leq-\chi\delta_n^2, 
\end{align*}
for all $n\geq n_0^{(6)}$, on $G_{2,n}$, we obtain
\begin{align*}
    {\bf{Y}}(\hat{\theta}_n+n^{-1/2}u)\leq-\chi\delta_n^2
\end{align*}
for all $u\in\hat{\mathbb{U}}_n\cap A_n^c$, which implies
\begin{align}
    n\sup_{u\in\hat{\mathbb{U}}_n\cap A_n^c}{\bf{Y}}(\hat{\theta}_n+n^{-1/2}u)\leq -\chi n\delta_n^2. \label{UAcP-2-2}
\end{align}
By (\ref{UAcP-2-1}) and (\ref{UAcP-2-2}), for all $n\geq n_0^{(6)}$, on $G_{2,n}$, 
\begin{align*}
    \sup_{u\in\hat{\mathbb{U}}_n\cap A_n^c}\log\hat{\bf{Z}}_n(u)&\leq Cn^{1/2}-Cn^{3/4}\\
    &\leq -Cn^{3/4}\bigl(1-n^{-1/4}\bigr).
\end{align*}
Since $n^{-1/4}\longrightarrow 0$ as $n\longrightarrow\infty$, there exists $n_0^{(7)}\in\mathbb{N}$ such that 
\begin{align*}
    n^{-1/4}\leq 1/2
\end{align*}
for all $n\geq n_0^{(7)}$, which deduces 
\begin{align*}
    \sup_{u\in\hat{\mathbb{U}}_n\cap A_n^c}\hat{\bf{Z}}_n(u)\leq\exp\Bigl\{-Cn^{3/4}\Bigr\}
\end{align*}
on $G_{2,n}$ for all $n\geq\max\{n_0^{(6)},n_0^{(7)}\}$. Therefore, for all $n\geq\max\{n_0^{(6)},n_0^{(7)}\}$, we have
\begin{align*}
    \int_{\hat{\mathbb{U}}_n\cap A_n^c}\hat{\bf{Z}}_n(u)du&\leq\sup_{u\in\hat{\mathbb{U}}_n\cap A_n^c}\hat{\bf{Z}}_n(u)\int_{\hat{\mathbb{U}}_n\cap A_n^c}du\\
    &\leq\exp\Bigl\{-Cn^{3/4}\Bigr\}\int_{\hat{\mathbb{U}}_n\cap A_n^c}du
\end{align*}
on $G_{2,n}$. Since $\Theta$ is bounded, there exists $M_0^{(3)}>0$ such that
\begin{align*}
    |\theta|\leq M_0^{(3)}
\end{align*}
for all $\theta\in\Theta$, and
\begin{align*}
    |n^{-1/2}u|&=\bigl|-\hat{\theta}_n+\hat{\theta}_n+n^{-1/2}u\bigr|\\
    &\leq|\hat{\theta}_n|+\bigl|\hat{\theta}_n+n^{-1/2}u\bigr|\leq 2M_0^{(3)},
\end{align*}
for all $u\in\hat{\mathbb{U}}_n\cap A_n^c$, so that 
\begin{align*}
    \hat{\mathbb{U}}_n\cap A_n^c\subset\Bigl\{u\in\mathbb{R}^q\big|\, |u|\leq 2n^{1/2}M_0^{(3)}\Bigr\}.
\end{align*}
In addition, by using the polar coordinates transformation, we have
\begin{align*}
    &\quad\ \int_{|u|\leq2n^{1/2}M_0^{(3)}}du\\
    &=\int_0^{2\pi}\cdots\int_0^{\pi}\int_0^{2n^{1/2}M_0^{(3)}}r^{q-1}\sin^{q-2}\phi_1\cdots\sin\phi_{q-2}drd\phi_1\cdots d\phi_{q-1}\\
    &\leq\int_0^{2\pi}\cdots\int_0^{\pi}\int_{0}^{2n^{1/2}M_0^{(3)}}r^{q-1}drd\phi_1\cdots d\phi_{q-1}\leq Cn^{q/2},
\end{align*}
where $r\in [0,\infty)$, $\phi_i\in[0,\pi]\ (i=1,\cdots,q-2)$, $\phi_{q-1}\in[0,2\pi]$ and
\begin{align*}
    u^{(1)}=r\cos\phi_1,\
    u^{(i)}=r\cos\phi_{i}\prod_{j=1}^{i-1}\sin\phi_{j}\ (i=2,\ldots, q-1),\ u^{(q)}=r\prod_{j=1}^{q-1}\sin\phi_{j}.
\end{align*}
Consequently, for all $n\geq\max\{n_0^{(6)},n_0^{(7)}\}$, we obtain 
\begin{align*}
    \int_{\hat{\mathbb{U}}_n\cap A_n^c}\hat{\bf{Z}}_n(u)du\leq Cn^{q/2}\exp\Bigl\{-Cn^{3/4}\Bigr\}
\end{align*}
on $G_{2,n}$. Since 
\begin{align*}
    Cn^{q/2}\exp\Bigl\{-Cn^{3/4}\Bigr\}\longrightarrow 0
\end{align*}
as $n\longrightarrow\infty$, there exists $n_0^{(8)}\in\mathbb{N}$ such that 
\begin{align*}
    Cn^{q/2}\exp\Bigl\{-Cn^{3/4}\Bigr\}<\varepsilon
\end{align*}
for all $n\geq n_0^{(8)}$, so that for all $n\geq\max\{n_0^{(6)},n_0^{(7)},n_0^{(8)}\}$,
\begin{align*}
    \int_{\hat{\mathbb{U}}_n\cap A_n^c}\hat{\bf{Z}}_n(u)du<\varepsilon
\end{align*}
on $G_{2,n}$. Therefore, (\ref{UAcP-1}) shows
\begin{align*}
    {\bf{P}}\biggl(\int_{\hat{\mathbb{U}}_n\cap A_n^c}\hat{\bf{Z}}_n(u)du>\varepsilon\biggr)&={\bf{P}}\biggl(\biggl\{\int_{\hat{\mathbb{U}}_n\cap A_n^c}\hat{\bf{Z}}_n(u)du>\varepsilon\biggr\}\cap G_{2,n}\biggr)\\
    &\qquad+{\bf{P}}\biggl(\biggl\{\int_{\hat{\mathbb{U}}_n\cap A_n^c}\hat{\bf{Z}}_n(u)du>\varepsilon\biggr\}\cap G^{c}_{2,n}\biggr)\\
    &\leq 0+{\bf{P}}\bigl(G^c_{2,n}\bigr)={\bf{P}}\bigl(G^c_{2,n}\bigr)<\varepsilon
\end{align*}
for all $n\geq\max\{n_0^{(5)}, n_0^{(6)},n_0^{(7)},n_0^{(8)}\}$, which yields (\ref{intAncprob}). This completes the proof.
\end{proof}
\begin{proof}[\textbf{Proof of Theorem \ref{Hexpansion}}]
By the change of variables
\begin{align*}
    \theta=\hat{\theta}_n+n^{-1/2}u
\end{align*}
for $u\in\hat{\mathbb{U}}_n$, it holds that
\begin{align*}
    \int_{\Theta}\exp\bigl\{{\bf{H}}_n(\theta)\bigr\}\pi(\theta)d\theta
    &=\int_{\hat{\mathbb{U}}_n}\exp\bigl\{{\bf{H}}_n(\hat{\theta}_n+n^{-1/2}u)\bigr\}\pi
    (\hat{\theta}_n+n^{-1/2}u)\Bigl|\det\Bigl(\frac{\partial\theta}{\partial u}\Bigr)\Bigr|du\\
    &=n^{-q/2}\int_{\hat{\mathbb{U}}_n}\exp\bigl\{{\bf{H}}_n(\hat{\theta}_n+n^{-1/2}u)\bigr\}\pi
    (\hat{\theta}_n+n^{-1/2}u)du
\end{align*}
since 
\begin{align*}
    \biggl|\det\Bigl(\frac{\partial\theta}{\partial u}\Bigr)\biggr|=\left|\det\begin{pmatrix}
    n^{-1/2} & 0 & \cdots & 0\\
    0 & n^{-1/2} & \cdots & 0\\
    \vdots & \vdots & \ddots & \vdots\\
    0 & 0 & \cdots & n^{-1/2}
    \end{pmatrix}\right|=n^{-q/2}.
\end{align*}
Hence, we have
\begin{align*}
    &\quad\ \log\int_{\Theta}\exp\bigl\{{\bf{H}}_n(\theta)\bigr\}\pi(\theta)d\theta\\
    &=\log n^{-q/2}\int_{\hat{\mathbb{U}}_n}\exp\bigl\{{\bf{H}}_n(\hat{\theta}_n+n^{-1/2}u)\bigr\}\pi
    (\hat{\theta}_n+n^{-1/2}u)du\\
    &=\log n^{-q/2}\exp\bigl\{{\bf{H}}_n(\hat{\theta}_n)\bigr\}\int_{\hat{\mathbb{U}}_n}\exp\bigl\{{\bf{H}}_n(\hat{\theta}_n+n^{-1/2}u)-{\bf{H}}_n(\hat{\theta}_n)\bigr\}\pi
    (\hat{\theta}_n+n^{-1/2}u)du\\
    &=-\frac{q}{2}\log n+{\bf{H}}_n(\hat{\theta}_n)+\log\int_{\hat{\mathbb{U}}_n}\hat{\bf{Z}}_n(u)
    \pi(\hat{\theta}_n+n^{-1/2}u)du,
\end{align*}
which yields 
\begin{align}
\begin{split}
    &\quad\ \frac{1}{n}\log\int_{\Theta}\exp\bigl\{{\bf{H}}_n(\theta)\bigr\}\pi(\theta)d\theta\\
    &=\frac{1}{n}{\bf{H}}_n(\hat{\theta}_n)-\frac{q}{2n}\log n+\frac{1}{n}\log\int_{\hat{\mathbb{U}}_n}\hat{\bf{Z}}_n(u)\pi(\hat{\theta}_n+n^{-1/2}u)du.\label{zenkin}
\end{split}
\end{align}
A decomposition is given by
\begin{align*}
    &\quad\ \int_{\hat{\mathbb{U}}_n}\hat{\bf{Z}}_n(u)\pi(\hat{\theta}_n+n^{-1/2}u)du\\
    &=\int_{\hat{\mathbb{U}}_n\cap A_n}\hat{\bf{Z}}_n(u)\pi(\hat{\theta}_n+n^{-1/2}u)du+\int_{\hat{\mathbb{U}}_n\cap A^c_n}
    \hat{\bf{Z}}_n(u)\pi(\hat{\theta}_n+n^{-1/2}u)du\\
    &=\int_{\hat{\mathbb{U}}_n\cap A_n}\hat{{\bf{Z}}}_n(u)\bigl\{\pi(\hat{\theta}_n+n^{-1/2}u)-\pi(\hat{\theta}_n)\bigr\}du\\
    &\qquad+\pi(\hat{\theta}_n)\int_{\hat{\mathbb{U}}_n\cap A_n}\hat{{\bf{Z}}}_n(u)du+\int_{\hat{\mathbb{U}}_n\cap A^c_n}
    \hat{\bf{Z}}_n(u)\pi(\hat{\theta}_n+n^{-1/2}u)du.
\end{align*}
By Lemma \ref{thetahat} and the continuous mapping theorem, we have
\begin{align*}
    \pi(\hat{\theta}_n)\stackrel{p}{\longrightarrow}\pi(\theta_0).
\end{align*}
In addition, it holds from Propositions \ref{Anprop1}-\ref{Ancprop} that 
\begin{align*}
    &\int_{\hat{\mathbb{U}}_n\cap A_n}\hat{{\bf{Z}}}_n(u)\bigl\{\pi(\hat{\theta}_n+n^{-1/2}u)-\pi(\hat{\theta}_n)\bigr\}du\stackrel{p}{\longrightarrow}0,\\
    &\qquad\qquad\int_{\hat{\mathbb{U}}_n\cap A_n}\hat{{\bf{Z}}}_n(u)du
    \stackrel{p}{\longrightarrow}\int_{\mathbb{R}^q}{\bf{Z}}(u)du
\end{align*}
and
\begin{align*}
    &\qquad\quad\int_{\hat{\mathbb{U}}_n\cap A^c_n}
    \hat{\bf{Z}}_n(u)\pi(\hat{\theta}_n+n^{-1/2}u)du\stackrel{p}{\longrightarrow}0.
\end{align*}
Consequently, Slutsky’s theorem implies 
\begin{align}
    \int_{\hat{\mathbb{U}}_n}\hat{\bf{Z}}_n(u)\pi(\hat{\theta}_n+n^{-1/2}u)du
    \stackrel{p}{\longrightarrow}\pi(\theta_0)\int_{\mathbb{R}^q}{\bf{Z}}(u)du.\label{intUprob}
\end{align}
Since $\pi(\theta_0)>0$ and ${\bf{\Gamma}}(\theta_0)$ is positive definite, we have
\begin{align*}
    &\quad\ \pi(\theta_0)\int_{\mathbb{R}^q}{\bf{Z}}(u)du
    =\pi(\theta_0)(2\pi)^{q/2}|\det{\bf{\Gamma}}(\theta_0)|^{-1/2}>0,
\end{align*}
so that the continuous mapping theorem and (\ref{intUprob}) imply
\begin{align*}
    \log\int_{\hat{\mathbb{U}}_n}\hat{\bf{Z}}_n(u)\pi(\hat{\theta}_n+n^{-1/2}u)du=O_p(1).
\end{align*}
Therefore, 
from (\ref{zenkin}), we obtain
\begin{align*}
    \frac{1}{n}\log\int_{\Theta}\exp\bigl\{{\bf{H}}_n(\theta)\bigr\}\pi(\theta)d\theta
    =\frac{1}{n}{\bf{H}}_n(\hat{\theta}_n)-\frac{q}{2n}\log n+O_p\Bigl(\frac{1}{n}\Bigr),
\end{align*}
which completes the proof.
\end{proof}
\begin{proof}[\textbf{Proof of Theorem \ref{BICcons2}}]
By Lemma 10 in Kusano and Uchida \cite{Kusano(jumpAIC)}, in an analogous manner to Theorem 3 in Kusano and Uchida \cite{Kusano(BIC)}, we can show the result.
\end{proof}
When Model $m^*$ is nested by Model $m$, there exist a matrix $F_{m^*,m}\in\mathbb{R}^{q_{m}\times q_{m^*}}$ with $F_{m^*,m}^{\top}F_{m^*,m}=\mathbb{I}_{q_{m^*}}$ and a constant $c_{m^*}\in\mathbb{R}^{q_{m^*}}$ such that
\begin{align}
    {\bf{H}}_{m^*,n}(\theta_{m^*})={\bf{H}}_{m,n}(F_{m^*,m}\theta_{m^*}+c_{m^*}) \label{nest}
\end{align}
for all $\theta_{m^*}\in\Theta_{m^*}$. For simplicity, 
we write $F_{m^*,m}$ and $c_{m^*}$ as $F$ and $c$. Set
\begin{align*}
    {\bf{F}}(\theta_{m^*})=F\theta_{m^*}+c
\end{align*}
for $\theta_{m^*}\in\Theta_{m^*}$ and
\begin{align*}
    {\bf{G}}_m(\theta_{m,0})=F\bigl(F^{\top} {\bf{\Gamma}}_m(\theta_{m,0})F\bigr)^{-1}F^{\top}-
    {\bf{\Gamma}}_m(\theta_{m,0})^{-1}.
\end{align*}
\begin{lemma}\label{Hsqrt}
Under $\bf{[A1]}$-$\bf{[A4]}$, for all $m\in\mathcal{M}$, 
\begin{align*}
    \frac{1}{\sqrt{n}}\partial_{\theta_{m}}
    {\bf{H}}_{m,n}(\theta_{m,0})\overset{d}{\longrightarrow}{\bf{\Gamma}}_m(\theta_{m,0})^{1/2}Z_{q_m}
\end{align*}
as $n\longrightarrow\infty$.
\end{lemma}
\begin{proof}
    See the proof of Theorem 2 in Kusano and Uchida \cite{Kusano(jump)}.
\end{proof}
\begin{proposition}\label{fthetahatprop}
Under $\bf{[A1]}$-$\bf{[A4]}$ and $\bf{[C1]}$, for all $m\in\mathcal{M}$,
\begin{align*}
    \sqrt{n}\bigl({\bf{F}}(\hat{\theta}_{m^*,n})-\hat{\theta}_{m,n}\bigr)
    \overset{d}{\longrightarrow}{\bf{G}}_m(\theta_{m,0})
     {\bf{\Gamma}}_m(\theta_{m,0})^{1/2}Z_{q_m}
\end{align*}
as $n\longrightarrow\infty$.
\end{proposition}
\begin{proof}
Since 
\begin{align*}
    \theta_{m,0}={\bf{F}}(\theta_{m^*,0}),
\end{align*}
we have
\begin{align*}
    {\bf{F}}(\hat{\theta}_{m^*,n})-\theta_{m,0}&={\bf{F}}(\hat{\theta}_{m^*,n})-{\bf{F}}(\theta_{m^*,0})\\
    &=\bigl(F\hat{\theta}_{m^*,n}+c\bigr)-\bigl(F\theta_{m^*,0}+c\bigr)=F\bigl(\hat{\theta}_{m^*,n}-\theta_{m^*,0}\bigr)
\end{align*}
and
\begin{align*}
    {\bf{F}}(\hat{\theta}_{m^*,n})-\hat{\theta}_{m,n}=\bigl({\bf{F}}(\hat{\theta}_{m^*,n})-\theta_{m,0}\bigr)-\bigl(\hat{\theta}_{m,n}-\theta_{m,0}\bigr),
\end{align*}
so that
\begin{align}
    \sqrt{n}\bigl({\bf{F}}(\hat{\theta}_{m^*,n})-\hat{\theta}_{m,n}\bigr)
    =F\hat{u}_{m^*,n}-\hat{u}_{m,n}.\label{thetadecom}
\end{align}
First, we consider the first term on the right-hand side of (\ref{thetadecom}). By using Taylor's theorem, 
\begin{align*}
    \frac{1}{\sqrt{n}}\partial_{\theta_{m^*}}
    {\bf{H}}_{m^*,n}(\hat{\theta}_{m^*,n})&=\frac{1}{\sqrt{n}}\partial_{\theta_{m^*}}
    {\bf{H}}_{m^*,n}(\theta_{m^*,0})\\
    &\quad-
    {\bf{\Gamma}}_{m^*,n}(\theta_{m^*,0})\hat{u}_{m^*,n}+{\bf{\bar{R}}}_{1,n}
\end{align*}
where 
\begin{align*}
    \check{\theta}_{m^*,\lambda,n}=\theta_{m^*,0}+\lambda(\hat{\theta}_{m^*,n}-\theta_{m^*,0})
\end{align*}
for $\lambda\in[0,1]$, and
\begin{align*}
    {\bf{\bar{R}}}_{1,n}^{(i)}=\sum_{j=1}^{q_{m^*}}\sum_{k=1}^{q_{m^*}}\biggl(\frac{1}{n^{3/2}}\int_0^1(1-\lambda)\partial_{\theta^{(i)}_{m^*}}\partial_{\theta^{(j)}_{m^*}}
    \partial_{\theta^{(k)}_{m^*}}{\bf{H}}_{m^*,n}(\check{\theta}_{m^*,\lambda,n})d\lambda\biggr)
    \hat{u}^{(j)}_{m^*,n}\hat{u}^{(k)}_{m^*,n}
\end{align*}
for $i=1,\ldots,q_{m^*}$. Since
\begin{align*}
    \frac{\partial {\bf{F}}(\theta_{m^*})}{\partial\theta_{m^*}^{\top}}=F,
\end{align*}
by the chain rule, it holds from (\ref{nest}) that 
\begin{align*}
    \partial_{\theta_{m^*}}{\bf{H}}_{m^*,n}(\theta_{m^*})
    &=F^{\top}\partial_{\theta_{m}}{\bf{H}}_{m,n}(\theta_{m})
\end{align*}
and
\begin{align*}
    \partial^2_{\theta_{m^*}}{\bf{H}}_{m^*,n}(\theta_{m^*})=F^{\top}\Bigl(\partial^2_{\theta_{m}}{\bf{H}}_{m,n}(\theta_{m})\Bigr)F.
\end{align*}
Hence, on $B_{1,n}$, it is shown that
\begin{align*}
    \frac{1}{\sqrt{n}}\partial_{\theta_{m^*}}
    {\bf{H}}_{m^*,n}(\theta_{m^*,0})=\biggl(-\frac{1}{n}\partial^2_{\theta_{m^*}}
    {\bf{H}}_{m^*,n}(\theta_{m^*,0})\biggr)\hat{u}_{m^*,n}-{\bf{\bar{R}}}_{1,n},
\end{align*}
which deduces 
\begin{align}
    F^{\top}\biggl(\frac{1}{\sqrt{n}}\partial_{\theta_{m}}
    {\bf{H}}_{m,n}(\theta_{m,0})\biggr)=\Bigl(F^{\top}{\bf{\Gamma}}_{m,n}(\theta_{m,0})F\Bigr)\hat{u}_{m^*,n}-{\bf{\bar{R}}}_{1,n},\label{FR}
\end{align}
where
\begin{align*}
    B_{1,n}=\Bigl\{\hat{\theta}_{m^*,n}\in\Theta_{m^*}\Bigr\}.
\end{align*}
Moreover, we set 
\begin{align*}
    \tilde{\bf{\Gamma}}_{m,n}
    &=\begin{cases}
    \displaystyle
     {\bf{\Gamma}}_{m,n}(\theta_{m,0}) & \bigl(on\, B_{2,n}\bigr),\\
    \mathbb{I}_{q_m} & \bigl(on\, B_{2,n}^c\bigr),
\end{cases}
\end{align*}
where
\begin{align*}
    B_{2,n}=\Bigl\{{\bf{\Gamma}}_{m,n}(\theta_{m,0})>0\Bigr\}.
\end{align*}
Note that $F^{\top}\tilde{\bf{\Gamma}}_{m,n}F$ is positive definite since $F$ is full column rank. On $B_{2,n}$, it holds from (\ref{FR}) that
\begin{align}
\begin{split}
    \hat{u}_{m^*,n}
    &=\bigl(F^{\top}\tilde{\bf{\Gamma}}_{m,n}F\bigr)^{-1}F^{\top}\biggl(\frac{1}{\sqrt{n}}\partial_{\theta_{m}}
    {\bf{H}}_{m,n}(\theta_{m,0})\biggr)
    +\bigl(F^{\top}\tilde{\bf{\Gamma}}_{m,n}F\bigr)^{-1}{\bf{\bar{R}}}_{1,n}.
\end{split}\label{thetamstar}
\end{align}
Next, we consider the second term on the right-hand side of (\ref{thetadecom}). On
\begin{align*}
    B_{3,n}=\Bigl\{\hat{\theta}_{m,n}\in\Theta_m\Bigr\},
\end{align*}
Taylor's theorem implies 
\begin{align*}
    \frac{1}{\sqrt{n}}\partial_{\theta_{m}}
    {\bf{H}}_{m,n}(\theta_{m,0})=
    {\bf{\Gamma}}_{m,n}(\theta_{m,0})\hat{u}_{m,n}-{\bf{\bar{R}}}_{2,n},
\end{align*}
where 
\begin{align*}
    \check{\theta}_{m,\lambda,n}=\theta_{m,0}+\lambda(\hat{\theta}_{m,n}-\theta_{m,0})
\end{align*}
for $\lambda\in[0,1]$, and
\begin{align*}
   {\bf{\bar{R}}}_{2,n}^{(i)}=\sum_{j=1}^{q_{m}}\sum_{k=1}^{q_{m}}\biggl(\frac{1}{n^{3/2}}\int_0^1(1-\lambda)\partial_{\theta^{(i)}_{m}}\partial_{\theta^{(j)}_{m}}
    \partial_{\theta^{(k)}_{m}}{\bf{H}}_{m,n}(\check{\theta}_{m,\lambda,n})d\lambda\biggr)
    \hat{u}^{(j)}_{m,n}\hat{u}^{(k)}_{m,n}
\end{align*}
for $i=1,\ldots,q_{m}$. Consequently, we obtain
\begin{align}
    \hat{u}_{m,n}=\tilde{\bf{\Gamma}}^{-1}_{m,n}\biggl(\frac{1}{\sqrt{n}}\partial_{\theta_{m}}
    {\bf{H}}_{m,n}(\theta_{m,0})\biggr)+\tilde{\bf{\Gamma}}^{-1}_{m,n}{\bf{\bar{R}}}_{2,n}\label{thetam}
\end{align}
on $B_{2,n}$. Therefore, 
it follows from (\ref{thetadecom}), (\ref{thetamstar}) and (\ref{thetam}) that
\begin{align*}
    \sqrt{n}\bigl({\bf{F}}(\hat{\theta}_{m^*,n})-\hat{\theta}_{m,n}\bigr)
    &=\tilde{F}_{m,n}
\end{align*}
on $B_{1,n}\cap B_{2,n}\cap B_{3,n}$, where
\begin{align*}
    \tilde{F}_{m,n}&=\Bigl(F\bigl(F^{\top}\tilde{\bf{\Gamma}}_{m,n}F\bigr)^{-1}F^{\top}-
    \tilde{\bf{\Gamma}}^{-1}_{m,n}\Bigr)
    \biggl(\frac{1}{\sqrt{n}}\partial_{\theta_{m}}
    {\bf{H}}_{m,n}(\theta_{m,0})\biggr)\\
    &\qquad\qquad\qquad\qquad\qquad
    +F\bigl(F^{\top}\tilde{\bf{\Gamma}}_{m,n}F\bigr)^{-1}{\bf{\bar{R}}}_{1,n}-\tilde{\bf{\Gamma}}^{-1}_{m,n}{\bf{\bar{R}}}_{2,n}.
\end{align*}
Since ${\bf{\Gamma}}(\theta_{m,0})$ is positive definite, Lemmas \ref{thetahat} and \ref{problemma} imply
\begin{align}
    {\bf{P}}\bigl(B_{1,n}\bigr)\longrightarrow 1,\quad 
    {\bf{P}}\bigl(B_{2,n}\bigr)\longrightarrow 1,\quad 
    {\bf{P}}\bigl(B_{3,n}\bigr)\longrightarrow 1 \label{Bn}
\end{align}
as $n\longrightarrow\infty$. As it holds from Lemmas \ref{thetahat} and \ref{problemma} that
\begin{align*}
     \tilde{\bf{\Gamma}}_{m,n}\overset{p}{\longrightarrow}{\bf{\Gamma}}_m(\theta_{m,0}),\quad 
     {\bf{\bar{R}}}_{1,n}\overset{p}{\longrightarrow}0,\quad {\bf{\bar{R}}}_{2,n}\overset{p}{\longrightarrow}0,
\end{align*}
Lemma \ref{Hsqrt} and Slutsky’s theorem yield
\begin{align}
    \tilde{F}_{m,n}&\overset{d}{\longrightarrow}
    {\bf{G}}_m(\theta_{m,0}){\bf{\Gamma}}_m(\theta_{m,0})^{1/2}Z_{q_m}.\label{Ftilde}
\end{align}
Therefore, for all closed set $C\in\mathbb{R}^{q_m}$, it follows from (\ref{Bn}) and (\ref{Ftilde}) that
\begin{align*}
    &\quad\ \limsup_{n\longrightarrow\infty}{\bf{P}}\Bigl(\sqrt{n}\bigl({\bf{F}}(\hat{\theta}_{m^*,n})-\hat{\theta}_{m,n}\bigr)\in C\Bigr)\\
    &= \limsup_{n\longrightarrow\infty}{\bf{P}}\Bigl(\Bigl\{\sqrt{n}\bigl({\bf{F}}(\hat{\theta}_{m^*,n})-\hat{\theta}_{m,n}\bigr)
    \in C\Bigr\}\cap (B_{1,n}\cap B_{2,n}\cap B_{3,n})\Bigr)\\
    &\qquad+\limsup_{n\longrightarrow\infty}{\bf{P}}\Bigl(\Bigl\{\sqrt{n}\bigl({\bf{F}}(\hat{\theta}_{m^*,n})-\hat{\theta}_{m,n}\bigr)
    \in C\Bigr\}\cap (B_{1,n}\cap B_{2,n}\cap B_{3,n})^c\Bigr)\\
    &\leq \limsup_{n\longrightarrow\infty}{\bf{P}}\Bigl(\bigl\{\tilde{F}_{m,n}\in C\bigr\}\cap (B_{1,n}\cap B_{2,n}\cap B_{3,n})\Bigr)+\limsup_{n\longrightarrow\infty}{\bf{P}}\bigl(B_{1,n}^c \cup B_{2,n}^c\cup B_{3,n}^c\bigr)\\
    &\leq \limsup_{n\longrightarrow\infty} {\bf{P}}\bigl(\tilde{F}_{m,n}\in C\bigr)+\limsup_{n\longrightarrow\infty}{\bf{P}}\bigl(B_{1,n}^c\bigr)+\limsup_{n\longrightarrow\infty}{\bf{P}}\bigl(B_{2,n}^c\bigr)+\limsup_{n\longrightarrow\infty}{\bf{P}}\bigl(B_{3,n}^c\bigr)\\
    &\leq {\bf{P}}\Bigl({\bf{G}}_m(\theta_{m,0}){\bf{\Gamma}}_m(\theta_{m,0})^{1/2}Z_{q_m}
    \in C\bigr),
\end{align*}
which deduces
\begin{align*}
    \sqrt{n}\bigl({\bf{F}}(\hat{\theta}_{m^*,n})-\hat{\theta}_{m,n}\bigr)\overset{d}{\longrightarrow}{\bf{G}}_m(\theta_{m,0}){\bf{\Gamma}}_m(\theta_{m,0})^{1/2}Z_{q_m}.
\end{align*}
This completes the proof.
\end{proof}
\begin{proof}[\textbf{Proof of Theorem \ref{QAIC}}]
By Taylor’s theorem, on $B_{3,n}$, we have
\begin{align*}
    {\bf{H}}_{m,n}\bigl({\bf{F}}(\hat{\theta}_{m^*,n})\bigr)&={\bf{H}}_{m,n}(\hat{\theta}_{m,n})
    +\bar{\bf{R}}_{3,m,m^*,n},
\end{align*}
so that
\begin{align*}
    {\bf{QAIC}}_n(m)-{\bf{QAIC}}_n(m^*)
    &=-2{\bf{H}}_{m,n}(\hat{\theta}_{m,n})+2{\bf{H}}_{m^*,n}\bigl(\hat{\theta}_{m^*,n}\bigr)
    +2(q_m-q_{m^*})\\
    &=-2{\bf{H}}_{m,n}(\hat{\theta}_{m,n})+2{\bf{H}}_{m,n}\bigl({\bf{F}}(\hat{\theta}_{m^*,n})\bigr)+2(q_m-q_{m^*})\\
    &=2\bar{\bf{R}}_{3,m,m^*,n}+2(q_m-q_{m^*}),
\end{align*}
where 
\begin{align*}
    \tilde{\theta}_{m,m^*,\lambda,n}=\hat{\theta}_{m,n}
    +\lambda\bigl({\bf{F}}(\hat{\theta}_{m^*,n})-\hat{\theta}_{m,n}\bigr)
\end{align*}
for $\lambda\in[0,1]$, and
\begin{align*}
    \bar{\bf{R}}_{3,m,m^*,n}=-\sqrt{n}\bigl({\bf{F}}(\hat{\theta}_{m^*,n})-\hat{\theta}_{m,n}\bigr)^{\top}
    \biggl\{\int_0^1(1-\lambda){\bf{\Gamma}}_{m,n}(\tilde{\theta}_{m,m^*,\lambda,n})d\lambda\biggr\}\sqrt{n}\bigl({\bf{F}}(\hat{\theta}_{m^*,n})-\hat{\theta}_{m,n}\bigr).
\end{align*}
As it holds from Lemmas \ref{thetahat} and \ref{problemma} that 
\begin{align*}
    \int_0^1(1-\lambda){\bf{\Gamma}}_{m,n}(\tilde{\theta}_{m,m^*,\lambda,n})d\lambda\overset{p}{\longrightarrow}\frac{1}{2} {\bf{\Gamma}}_m(\theta_{m,0})
\end{align*}
in a similar manner to the proof of Theorem 2 in Kusano and Uchida \cite{Kusano(jump)}, it follows from Proposition \ref{fthetahatprop} and the continuous mapping theorem that
\begin{align*}
    &2\sqrt{n}\bigl({\bf{F}}(\hat{\theta}_{m^*,n})-\hat{\theta}_{m,n}\bigr)^{\top}
    \biggl\{\int_0^1(1-\lambda){\bf{\Gamma}}_{m,n}(\tilde{\theta}_{m,m^*,\lambda,n})d\lambda\biggr\}\sqrt{n}\bigl({\bf{F}}(\hat{\theta}_{m^*,n})-\hat{\theta}_{m,n}\bigr)\\
    &\qquad\qquad\qquad\qquad\overset{d}{\longrightarrow}Z_{q_m}^{\top} {\bf{P}}_m(\theta_{m,0})Z_{q_m},
\end{align*}
where
\begin{align*}
    {\bf{P}}_m(\theta_{m,0})={\bf{\Gamma}}_m(\theta_{m,0})^{1/2}{\bf{G}}_m(\theta_{m,0}){\bf{\Gamma}}_m(\theta_{m,0})
    {\bf{G}}_m(\theta_{m,0}){\bf{\Gamma}}_m(\theta_{m,0})^{1/2}.
\end{align*}
Since ${\bf{P}}_m(\theta_{m,0})$ is an orthogonal projection matrix and
\begin{align*}
    \rank{{\bf{P}}_m(\theta_{m,0})}=q_m-q_{m^*}, 
\end{align*}
we have
\begin{align*}
    Z_{q_m}^{\top} {\bf{P}}_m(\theta_{m,0})Z_{q_m}\sim\chi^2_{q_m-q_{m^*}}.
\end{align*}
Consequently, since
\begin{align*}
    -2\bar{\bf{R}}_{3,m,m^*,n}\overset{d}{\longrightarrow}\chi^2_{q_m-q_{m^*}},
\end{align*}
(\ref{Bn}) implies
\begin{align*}
    &\quad \lim_{n\longrightarrow\infty} {\bf{P}}\Bigl({\bf{QAIC}}_n(m)<{\bf{QAIC}}_n(m^*)\Bigr)\\
    &=\lim_{n\longrightarrow\infty} {\bf{P}}\Bigl(\Bigl\{{\bf{QAIC}}_n(m)<{\bf{QAIC}}_n(m^*)\Bigr\}\cap B_{3,n}\Bigr)\\
    &\qquad\qquad+\lim_{n\longrightarrow\infty} {\bf{P}}\Bigl(\Bigl\{{\bf{QAIC}}_n(m)<{\bf{QAIC}}_n(m^*)\Bigr\}\cap B_{3,n}^c\Bigr) \\
    &=\lim_{n\longrightarrow\infty}{\bf{P}}\Bigl(\Bigl\{-2\bar{\bf{R}}_{3,m,m^*,n}>2(q_m-q_{m^*})\Bigr\}\cap B_{3,n}\Bigr)+0\\
    &=\lim_{n\longrightarrow\infty}{\bf{P}}\Bigl(-2\bar{\bf{R}}_{3,m,m^*,n}>2(q_m-q_{m^*})\Bigr)\\
    &={\bf{P}}\Bigl(\chi^2_{q_m-q_{m^*}}>2(q_{m}-q_{m^*})\Bigr)>0,
\end{align*}
which completes the proof.
\end{proof}
\begin{proof}[\textbf{Proof of Theorem \ref{BICcons1}}]
By using Lemmas \ref{thetahat} and \ref{problemma}, in a similar way to Theorem 2 in Kusano and Uchida \cite{Kusano(BIC)}, this result can be shown.
\end{proof}

\end{document}